\documentclass[11pt,reqno]{amsart}

\usepackage{enumerate}
\usepackage{amsmath, amssymb, amsthm}
\usepackage{mathrsfs}
\usepackage{esint}
\usepackage{xcolor}
\usepackage{mathtools}
\usepackage{hyperref}
\usepackage{bm}
\usepackage{kotex}

\numberwithin{equation}{section}

\newtheorem{theorem}{Theorem}[section]
\newtheorem{lemma}[theorem]{Lemma}
\newtheorem{corollary}[theorem]{Corollary}
\newtheorem{remark}[theorem]{\bf{Remark}}
\newtheorem{assumption}[theorem]{Assumption}

\newtheorem{definition}[theorem]{Definition}

\theoremstyle{remark}
\theoremstyle{definition}

\newcommand\bL{\mathbb{L}}
\newcommand\bR{\mathbb{R}}

\newcommand\bH{\mathbb{H}}

\newcommand\bZ{\mathbb{Z}}

\newcommand\bE{\mathbb{E}}
\newcommand\bN{\mathbb{N}}

\newcommand\cB{\mathcal{B}}
\newcommand\cC{\mathcal{C}}

\newcommand\cF{\mathcal{F}}

\newcommand\cH{\mathcal{H}}

\newcommand\cL{\mathcal{L}}
\newcommand\cP{\mathcal{P}}
\newcommand\cR{\mathcal{R}}
\newcommand\cS{\mathcal{S}}
\newcommand\cM{\mathcal{M}}

\newcommand\cO{\mathcal{O}}

\newcommand\cbrk{\text{$]$\kern-.15em$]$}}
\newcommand\opar{\text{\,\raise.2ex\hbox{${\scriptstyle
|}$}\kern-.34em$($}}
\newcommand\cpar{\text{$)$\kern-.34em\raise.2ex\hbox{${\scriptstyle |}$}}\,}

\newcommand\ep{\varepsilon}

\begin{document}

\title[\texorpdfstring{$L_p$}{Lg}-regularity theory for semilinear SPDEs with time white noise]{
\texorpdfstring{$L_p$}{Lg}-regularity theory for semilinear stochastic partial differential equations with multiplicative white noise
}

\author{Beom-Seok Han}

\address{Department of Mathematics, Pohang University of Science and Technology, 77, Cheongam-ro, Nam-gu, Pohang, Gyeongbuk, 37673, Republic of Korea}

\email{hanbeom@postech.ac.kr}


\thanks{This work was supported by the National Research Foundation of Korea (NRF) grant
funded by the Korea government (MSIT) (No. NRF-2021R1C1C2007792) and the BK21 FOUR (Fostering Outstanding Universities for Research) funded by the Ministry of Education (MOE, Korea) and National Research Foundation of Korea (NRF)}




\begin{abstract}

We establish the $L_p$-regularity theory for a semilinear stochastic partial differential equation with multiplicative white noise:
\begin{equation*}
du = \left(a^{ij}u_{x^ix^j} + b^{i}u_{x^i} + cu + \bar b^{i}|u|^\lambda u_{x^i}\right)dt + \sigma^k(u)dw_t^k,\quad (\omega,t,x)\in\Omega\times(0,\infty)\times\bR^d
\end{equation*}
with $u(0,\cdot) = u_0$, where $\lambda>0$, the set $\{ w_t^k,k=1,2,\dots \}$ is a set of one-dimensional independent Wiener processes, and the function $u_0 = u_0(\omega,x)$ is nonnegative random initial data. The coefficients $a^{ij},b^i,c$ depend on $(\omega,t,x)$, and $\bar b^i$ depends on \linebreak $(\omega,t,x^1,\dots,x^{i-1},x^{i+1},\dots,x^d)$. The coefficients $a^{ij},b^i,c,\bar{b}^i$ are uniformly bounded and twice continuously differentiable. The leading coefficient $a$ satisfies the ellipticity condition. Depending on the diffusion coefficient $\sigma^k(u) = \sigma^k(\omega,t,x,u)$, we consider two cases:
\begin{enumerate}[(i)]
\item 
$\lambda\in(0,\infty)$ and $\sigma^k(u)$ has a Lipschitz continuity and linear growth in $u$.

\item
$\lambda,\lambda_0\in(0,1/d)$ and $\sigma^k(\omega,t,x,u) = \mu^k(\omega,t,x) |u|^{1+\lambda_0}$ with $\left| \sum_k|\mu^k(\omega,t,\cdot)|^2 \right|_{C(\bR^d)}<\infty$ for all $(\omega,t)\in\Omega\times[0,\infty)$.
\end{enumerate}
We obtain the existence, uniqueness, $L_p$ regularity, and H\"older regularity of the solution. It should be noted that each case has different regularity results. For example, in the case of $(i)$, for $\ep>0$
\begin{equation*}
u \in C^{1/2 - \ep,1 - \ep}_{t,x}([0,T]\times\bR^d)\quad \forall T<\infty,  
\end{equation*}
almost surely. On the other hand, in the case of $(ii)$, if $\lambda,\lambda_0\in(0,1/d)$, for $\ep>0$
\begin{equation*}
u \in C^{\frac{1-(\lambda d) \vee (\lambda_0 d)}{2} - \ep,1-(\lambda d) \vee (\lambda_0 d) - \ep}_{t,x}([0,T]\times\bR^d)\quad \forall T<\infty
\end{equation*}
almost surely.
It should be noted that  $\lambda$ can be any positive number and that the solution regularity is independent of the nonlinear terms in case $(i)$. In case $(ii)$, however, $\lambda$ and $\lambda_0$ should satisfy $\lambda,\lambda_0\in(0,1/d)$, and the regularities of the solution are affected by $\lambda,\lambda_0$ and $d$. This difference is due to the need to use a different proof method for the super-linear diffusion coefficient $\sigma^k(u)$.

\end{abstract}

\maketitle

\section{Introduction}

The purpose of this paper is to obtain the uniqueness, existence, $L_p$-regularity, and H\"older regularity of a solution to a semilinear stochastic partial differential equation with the nonlinear diffusion coefficients $\sigma^k(u)$:
\begin{equation}
\label{main_equation}
du = \left(a^{ij}u_{x^ix^j} + b^i u_{x^i} + cu + \bar{b}^i |u|^{\lambda} u_{x^i}\right)dt + \sigma^k(u)dw_t^k,\quad (\omega,t,x)\in\Omega\times(0,\infty)\times\bR^d
\end{equation}
with $\quad u(0,\cdot) = u_0(\cdot)$, where $\lambda > 0$ and $\{w_t^k:k\in\bN\}$ is a set of one-dimensional independent Wiener processes. The coefficients $a^{ij}(\omega,t,x)$, $b^i(\omega,t,x)$, and $c(\omega,t,x)$ are $\cP\times\cR(\bR^d)$-measurable, and $\bar b^i = \bar b^i(\omega,t,x^1,\dots,x^{i-1},x^{i+1},x^d)$ is $\cP\times\cB\left(\bR^{d-1}\right)$-measurable. The coefficients $a^{ij},b^i,c,$ and $\bar b^i$ are uniformly bounded and twice continuously differentiable. The coefficient $\sigma^k(u) = \sigma^k(\omega,t,x,u)$ is $\cP\times\cB(\bR^d)\times\cB(\bR)$-measurable. Depending on the nonlinearity types of $\sigma^k(u)$, different conditions on $\lambda$ are required.
\begin{enumerate}[(i)]
\item 
If $\sigma^k(u)$ has a Lipschitz continuity and linear growth in $u$, the $L_p$-regularity results are obtained for $\lambda\in(0,\infty)$.

\item
If $\sigma^k(u) = \mu^k(\omega,t,x,u) |u|^{1+\lambda_0}$ with $\left| \sum_k|\mu^k(t,\cdot)|^2 \right|_{C(\bR^d)}<\infty$ for all $(\omega,t)\in\Omega\times[0,\infty)$, the $L_p$-regularity results are obtained for $\lambda,\lambda_0\in(1/d)$.

\end{enumerate}

To motivate this problem, we provide historical notes. 
Consider a stochastic semilinear partial differential equation
\begin{equation}
\label{main_equation_general_form}
u_t = \cL u + (g(u))_{x}  + \sigma(u)\dot W,\quad (t,x)\in(0,\infty)\times\cO\,; \quad u(0,\cdot) = u_0(\cdot),
\end{equation}
where $\cL$ is a second order differential operator, $g(u)$ and $\sigma(u)$ are nonlinear functions in $u$, $\dot W$ is white noise, and $\cO\subset \bR^d$. 

One of the most well-known examples of \eqref{main_equation_general_form} is a generalized Burgers' equation. \cite{crighton1992modern,ladyzhenskaia1968linear,dix1996nonuniqueness,bekiranov1996initial,tersenov2010generalized} address the case of a deterministic equation (e.g., $g(u) = |u|^{1+\lambda}$, $\sigma(u) = 0$). In particular, the uniqueness and existence of the local solution are proven in \cite{dix1996nonuniqueness}, and the unique existence of a global solution with small initial data is obtained in \cite{bekiranov1996initial}. In \cite{tersenov2010generalized}, the unique solvability is obtained under dominating conditions on $g$. As a stochastic counterpart, stochastic semilinear partial differential equations are studied  in \cite{bertini1994stochastic,da1994stochastic,da1995stochastic,gyongy1998existence,gyongy1999stochastic,gyongy2000lp,leon2000stochastic,englezos2013stochastic,catuogno2014strong,lewis2018stochastic}. The uniqueness and existence of the $L_p$-valued solution on the unit interval are proven in \cite{gyongy1999stochastic}. Additionally, the results are extended to a bounded convex domain with a smooth boundary on $\bR^d$ in \cite{gyongy2000lp}. For more information, see \cite{rockner2006kolmogorov} and the references therein. 

As another example of \eqref{main_equation_general_form}, we consider a stochastic partial differential equation driven by multiplicative noise:
\begin{equation}
\label{second_order_SPDE_example_1}
u_t = \cL u + \xi|u|^{\gamma} \dot W,\quad (\omega,t,x)\in\Omega\times(0,\infty)\times\cO \,;\quad u(0,\cdot) = u_0,
\end{equation}
where $\gamma>0$, $\dot W$ is the space-time white noise, and $\cO$ is either $\bR$ or $(0,1)$; see \cite{Kijung,kry1999analytic,krylov1997result,krylov1997spde,mueller1991long,mueller2014nonuniqueness,mytnik1998weak,mytnik2011pathwise,walsh1986introduction,Xiong,choi2021regularity,han2019boundary}. \cite{kry1999analytic,krylov1997result,krylov1997spde,mueller1991long,choi2021regularity,han2019boundary} contain especially relevant results related to stochastic partial differential equations with super-linear diffusion coefficients  (e.g., $\gamma > 1$). In \cite{mueller1991long,kry1999analytic,han2019boundary}, equation \eqref{second_order_SPDE_example_1} with $\gamma\in [1,3/2)$ is considered on a bounded domain or real line. In \cite{mueller1991long}, the long-time existence of a mild solution is proven for equation \eqref{second_order_SPDE_example_1}, where  $\cL = \Delta$, $\xi = 1$, $\cO = (0,1)$, and $u(0) = u_0$ is deterministic, continuous, and nonnegative. On the other hand, in \cite[Section 8.4]{kry1999analytic}, the existence of a solution in $L_p(\bR)$ spaces is obtained with a second-order differential operator $\cL = aD^2 + bD + c$, where the bounded coefficients $a,b,c,\xi$ depend on $(\omega,t,x)$. Additionally, in \cite{han2019boundary}, the regularity and boundary behavior of a solution are achieved with $\cO = (0,1)$ and $\cL = aD^2+bD+c$, where $a,b,c,\xi$ are random as well as space and time-dependent coefficients, and the nonnegative random initial data $u(0) = u_0$. In the case of $\gamma > 3/2$, \cite{mueller2000critical} shows that if $\cL = \Delta$, $\xi = 1$, and the initial data $u_0$ is nonnegative, nontrivial, and vanishing at the endpoints, then there is a positive probability that a solution blows up in finite time.

Therefore, the aim of this paper is to obtain the uniqueness, existence, $L_p$ regularity, and H\"older regularity of the \textit{strong} solution to equation \eqref{main_equation} by establishing the $L_p$ regularity theory. The first novelty of our result is the initial data $u_0$ and that the coefficients $a,b,c,\bar b$ are assumed to be \textit{random} functions. In other words, we establish the $L_p$ regularity theory for \eqref{main_equation} with the {\it random} initial data $u_0(\omega, x)$ and the {\it random} coefficients $a^{ij}(\omega,t,x)$, $b^i(\omega,t,x)$, $c(\omega,t, x)$, $\bar{b}^i(\omega,t,\bar x^i)$, and $\sigma(\omega,t,x,u)$, where $\bar x^i = (x^1,\dots,x^{i-1},x^{i+1},\dots,x^d)$. Employing randomness is one of the notable  advantages of Krylov's $L_p$ theory when compared with the results of the stochastic generalized Burgers' equation, where the coefficients are required to be constant.

Second, the relations between the nonlinear terms and the regularity of the solution are suggested. It should be noted that if $\sigma^k(u)$ has a Lipschitz continuity and linear growth in $u$, the regularity of the solution is independent of the nonlinear terms. For example, if $\sigma^k(u)$ has a Lipschitz continuity and linear growth in $u$, for $\lambda\in (0,\infty)$ and nonnegative initial data $u_0\in U_p^1\cap L_{p(1+\lambda)}(\Omega\times\bR^d)$ ( see Definition \ref{definition_of_function_spaces} for the definition of $U_p^1$), equation \eqref{main_equation} has a unique solution $u\in\cH_{p,loc}^1$ (see Definition \ref{definition_of_sol_space}). In addition, for small $\ep>0$
\begin{equation*}
\sup_{t\leq T}|u(\omega,t,\cdot)|_{C^{1-\ep}(\bR^d)} + \sup_{x\in\bR^d}|u(\omega,\cdot,x)|_{C^{1/2-\ep}([0,T])}<\infty\quad \forall T<\infty.
\end{equation*} 
almost surely.
On the other hand, if the diffusion coefficient $\sigma^k(u)$ is super-linear in $u$, the regularity of the solution is affected by the nonlinear terms $|u|^\lambda u_{x^i}$ and $\sigma^k(u)$. As an example, if we consider that $\sigma^k(u) = \mu^k |u|^{1+\lambda_0}$ with $\left|\sum_k |\mu^k(\omega,t,\cdot)|^2\right|_{C(\bR^d)}<\infty$ almost surely, then for $\lambda,\lambda_0\in(0,1/d)$, and nonnegative initial data $u_0\in U_p^{1-(\lambda d)\vee (\lambda_0 d)}\cap L_1(\Omega\times\bR^d)$, equation \eqref{main_equation} has a unique solution $u\in\cH_{p,loc}^{1-(\lambda d)\vee (\lambda_0 d)}$. Furthermore, for small $\ep>0$
\begin{equation*}
\sup_{t \leq T}|u(\omega,t,\cdot)|_{C^{1 - (\lambda d) \vee (\lambda_0 d) - \ep}(\bR^d)} + \sup_{x\in\bR^d}|u(\omega,\cdot,x)|_{C^{\frac{1-(\lambda d) \vee (\lambda_0 d)}{2}-\ep}([0,T])}<\infty \quad \forall T<\infty
\end{equation*}
almost surely. It should be noted that the parameters $\lambda,\lambda_0$ and $d$ have an effect on the solution regularity if $\sigma^k(u)$ is a super-linear function in $u$.  We discuss the regularity and its dependency in Remarks \ref{remark:condition_on_lambda_1}, \ref{remark:condition_on_lambda_2}, and \ref{remark:condition_on_lambda_3}.

Finally, we achieve a sufficient condition for the unique solvability in $L_p$ spaces. For example, if $\sigma^k(u)$ has a Lipschitz continuity and linear growth in $u$, $\lambda$ can be any positive real number. 
However, if $\sigma^k(u) = \mu^k(\omega,t,x)|u|^{1+\lambda_0}$ with $\left| \sum_k|\mu^k(\omega,t,\cdot)|^2 \right|_{C(\bR^d)}<\infty$ for all $(\omega,t)\in\Omega\times[0,\infty)$, algebraic conditions on $\lambda$ and $\lambda_0$ are required: $\lambda,\lambda_0\in (1/d)$. Indeed, in the proof of any cases, we need an $L_q$ (for some $q\geq1$) bound of the solution to address the nonlinear terms. If $\sigma^k(u)$ has a Lipschitz continuity and linear growth in $u$, we apply It\^o's formula to obtain the $L_q$ $(q>1)$ bound and the nonlinear term vanishes using the fundamental theorem of calculus since $u^\lambda u_{x^i} = \frac{1}{1+\lambda}\left( u^{1+\lambda} \right)_{x^i}$. Therefore, only $\lambda>0$ is needed. In the case of $\sigma(u) = |u|^{1+\lambda_0}$ with $\lambda_0>0$, however, we should employ a uniform $L_1(\bR^d)$ bound for the solution to control both of the nonlinear terms, $u^\lambda u_x$ and $|u|^{1+\lambda_0}$. Since we interpret the nonlinear terms as
$$ u^\lambda u_{x} = \frac{1}{1+\lambda} \left( |u|^{1+\lambda} \right)_{x} = \frac{1}{1+\lambda} \left( |u|^{\lambda}\cdot |u| \right)_{x}\quad\text{and}\quad |u|^{1+\lambda_0} = |u|^{\lambda_0}\cdot |u|,
$$
if $\lambda,\lambda_0\in(0,1/d)$, then the surplus terms $|u|^{\lambda}$ and $|u|^{\lambda_0}$ can be controlled using the uniform $L_1(\bR^d)$ bound of the solution; see Remarks \ref{remark:condition_on_lambda_2} and \ref{remark:condition_on_lambda_3}.

This paper is organized as follows: Section \ref{sec:preliminaries} introduces the preliminary definitions and properties. In Section \ref{sec:main_results}, we provide the existence, uniqueness, $L_p$-regularity, and H\"older regularity of a solution to equation \eqref{main_equation}. The proofs of the main results are contained in Sections \ref{proof_of_theorem_white_noise_in_time_large_lambda} and \ref{proof_of_theorem_white_noise_in_time_small_lambda}. The basic strategy of the proof is to control the solution $u$ using the initial data $u_0$. In detail, the solution estimate required to prove the existence of a global solution varies depending on the nonlinearity of the diffusion coefficient $\sigma(u)$. If $\sigma(u)$ has a Lipschitz continuity and linear growth in $u$, we obtain the $L_{p(1+\lambda)}$ bound of the solution by employing Ito's formula and the fundamental theorem of calculus. With the help of the $L_{p(1+\lambda)}$ bound of the solution, we can show that there exists a global solution by contradiction. In the case that $\sigma(u)$ is a super-linear function in $u$ (e.g., $\sigma(u) = |u|^{1+\lambda_0}$), we prove that the uniform $L_1$ norm of the solution is controlled by the initial data $u_0$. Using the uniform $L_1$ bound of the solutions, we can extend the local existence time.


We finish introduction by providing the notations. Let $\bN$ and $\bR$ denote the set of natural numbers and real numbers, respectively. We use $:=$ to denote a definition.  For a real-valued function $f$, we set
\begin{equation*}
f^+:= \frac{f+|f|}{2},\quad f^- = -\frac{f-|f|}{2}
\end{equation*} For a normed space $F$, a measure space $(X,\mathcal{M},\mu)$, and $p\in [1,\infty)$, a space $L_{p}(X,\cM,\mu;F)$ is a set of $F$-valued $\mathcal{M}^{\mu}$-measurable functions such that
\begin{equation}
\label{def_norm}
\| u \|_{L_{p}(X,\cM,\mu;F)} := \left( \int_{X} \| u(x) \|_{F}^{p}\mu(dx)\right)^{1/p}<\infty.
\end{equation}
A set $\mathcal{M}^{\mu}$ is the completion of $\cM$ with respect to the measure $\mu$.
For $\alpha\in(0,1]$ and $T>0$, a set $C^{\alpha}([0,T];F)$ is the set of $F$-valued continuous functions $u$ such that
$$ |u|_{C^{\alpha}([0,T];F)}:=\sup_{t\in[0,T]}|u(t)|_{F}+\sup_{\substack{s,t\in[0,T], \\ s\neq t}}\frac{|u(t)-u(s)|_F}{|t-s|^{\alpha}}<\infty.$$
For $a,b\in \bR$, set
$a \wedge b := \min\{a,b\}$, and $a \vee b := \max\{a,b\}$.
Let $\cS = \cS(\bR^d)$ denote the set of Schwartz functions on $\bR^d$. The Einstein's summation convention with respect to $i,j$, and $k$ is assumed. A generic constant $N$ can be different from line to line. The constant $N=N(a,b,\ldots)$ denotes that $N$ depends only on $a,b,\ldots$. For functions depending on $\omega$, $t$, and $x$, the argument $\omega \in \Omega$ is usually omitted.

\section{Preliminaries} 
\label{sec:preliminaries}

This section reviews the definitions and properties of stochastic Banach spaces. The space $\cH_p^\gamma(\tau)$ is used as solution spaces. Gr\"onwall's inequality and the H\"older embedding theorems for $\cH_p^\gamma(\tau)$ are employed to show the unique solvability and solution regularity. For more information, see \cite{kry1999analytic,grafakos2009modern,krylov2008lectures}.

\begin{definition}[Bessel potential space]
Let $p>1$ and $\gamma \in \mathbb{R}$. The space $H_p^\gamma=H_p^\gamma(\bR^d)$ is the set of all tempered distributions $u$ on $\bR^d$ satisfying

$$ \| u \|_{H_p^\gamma} := \left\| (1-\Delta)^{\gamma/2} u\right\|_{L_p} = \left\| \cF^{-1}\left[ \left(1+|\xi|^2\right)^{\gamma/2}\cF(u)(\xi)\right]\right\|_{L_p}<\infty.
$$
Similarly, $H_p^\gamma(\ell_2) = H_p^\gamma(\bR^d;\ell_2)$ is a space of $\ell_2$-valued functions $g=(g^1,g^2,\cdots)$ satisfying
$$ \|g\|_{H_{p}^\gamma(\ell_2)}:= \left\| \left| \left(1-\Delta\right)^{\gamma/2} g\right|_{l_2}\right\|_{L_p} = \left\| \left|\cF^{-1}\left[ \left(1+|\xi|^2\right)^{\gamma/2}\cF(g)(\xi)\right]\right|_{\ell_2} \right\|_{L_p}
< \infty. 
$$
For $\gamma = 0$, we set $L_p := H_p^0$ and $L_p(\ell_2) := H_p^0(\ell_2)$.
\end{definition}

\begin{remark} \label{Kernel}
It is well-known that the Bessel potential operator $(1-\Delta)^{-\gamma/2}$  has a representation for $\gamma\in(0,\infty)$. Precisely, for $\gamma\in (0,\infty)$ and $u\in \cS$, we have
\begin{equation*}
(1-\Delta)^{-\gamma/2}u(x)=\int_{\bR^d}R_{\gamma}(x-y)u(y)dy,
\end{equation*}
where 
\begin{equation*} 
|R_\gamma(x)| \leq N(\gamma,d)\left(e^{-\frac{|x|}{2}}1_{|x|\geq2} + A_\gamma(x)1_{|x|<2}\right)
\end{equation*}
and
\begin{equation*}
\begin{aligned}
A_{\gamma}(x):=
\begin{cases}
|x|^{\gamma-d} + 1 + O(|x|^{\gamma-d+2}) \quad &\mbox{if} \quad 0<\gamma<d,\\ 
\log(2/|x|) + 1 + O(|x|^{2}) \quad &\mbox{if} \quad \gamma=d,\\ 
1 + O(|x|^{\gamma-d}) \quad &\mbox{if} \quad \gamma>d.
\end{cases}
\end{aligned}
\end{equation*}
For more information, see \cite[Proposition 1.2.5.]{grafakos2009modern}.

\end{remark}

Below, we introduce the space of the point-wise multipliers for $H_p^\gamma$.

\begin{definition}
\label{def_pointwise_multiplier}
Fix $\gamma\in\bR$ and $\alpha\in[0,1)$ such that $\alpha = 0$ if $\gamma\in\bZ$ and $\alpha>0$ if $|\gamma|+\alpha$ is not an integer. Define
\begin{equation*}
\begin{aligned}
B^{|\gamma|+\alpha} = 
\begin{cases}
B(\bR^d) &\quad\text{if } \gamma = 0, \\
C^{|\gamma|-1,1}(\bR^d) &\quad\text{if $\gamma$ is a nonzero integer}, \\
C^{|\gamma|+\alpha}(\bR^d) &\quad\text{otherwise};
\end{cases}
\end{aligned}
\end{equation*}
where $B(\bR^d)$ is the space of the bounded Borel functions on $\bR^d$, $C^{|\gamma|-1,1}(\bR^d)$ is the space of $|\gamma|-1$ times the continuously differentiable functions whose derivatives of the $(|\gamma|-1)$-th order derivative are Lipschitz continuous, and $C^{|\gamma|+\alpha}$ is real-valued H\"older spaces. The space $B(\ell_2)$ denotes a function space with hte $\ell_2$-valued functions instead of the real-valued function spaces.

\end{definition}

Next, we gather the properties of the Bessel potential space $H_p^\gamma$.

\begin{lemma}
\label{prop_of_bessel_space} 
Let $p>1$ and $\gamma \in \bR$. 
\begin{enumerate}[(i)]
\item 
\label{dense_subset_bessel_potential}
The space  $C_c^\infty(\bR^d)$ is dense in $H_{p}^{\gamma}$. 

\item
\label{sobolev-embedding} 
Let $\gamma - d/p = n+\nu$ for some $n=0,1,\cdots$ and $\nu\in(0,1]$. Then, for any  $i\in\{ 0,1,\cdots,n \}$, we have
\begin{equation} 
\label{holder embedding}
\left| D^i u \right|_{C(\bR^d)} + \left[D^n u\right]_{\cC^\nu(\bR^d)} \leq N \| u \|_{H_{p}^\gamma},
\end{equation}
where $N = N(d,p,\gamma)$ and $\cC^\nu$ is a Zygmund space.

\item
\label{bounded_operator}
The operator $D_i:H_p^{\gamma}\to H_p^{\gamma+1}$ is bounded. Moreover, for any $u\in H_p^{\gamma+1}$,
$$ \left\| D^i u \right\|_{H_p^\gamma} \leq N\| u \|_{H_p^{\gamma+1}},
$$
where $N = N(d,p,\gamma)$.

\item
\label{norm_bounded}
Let $\mu\leq\gamma$ and $u\in H_p^\gamma$. Then, $u\in H_p^\mu$ and 
$$ \| u \|_{H_p^\mu} \leq \| u \|_{H_p^\gamma}. 
$$

\item 
\label{iso} (isometry). For any $\mu,\gamma\in\bR$, the operator $(1-\Delta)^{\mu/2}:H_p^\gamma\to H_p^{\gamma-\mu}$ is an isometry.

\item
\label{multi_ineq} (multiplicative inequality). Let 
\begin{equation*} \label{condition_of_constants_interpolation}
\begin{gathered}
\ep\in[0,1],\quad p_i\in(1,\infty),\quad\gamma_i\in \bR,\quad i=0,1,\\
\gamma=\ep\gamma_0+(1-\ep)\gamma_1,\quad1/p=\ep/p_0+(1-\ep)/p_1.
\end{gathered}
\end{equation*}
Then, we have
\begin{equation}
\label{eq:multi_ineq 1}
\|u\|_{H^\gamma_{p}} \leq \|u\|^{\ep}_{H^{\gamma_0}_{p_0}}\|u\|^{1-\ep}_{H^{\gamma_1}_{p_1}}.
\end{equation}
In particular, for any $\delta>0$,
\begin{equation}
\label{eq:multi_ineq 2}
\|u\|_{H^\gamma_{p}} \leq \delta\|u\|_{H^{\gamma_0}_{p_0}} + \delta^{-\frac{\ep}{1-\ep}}\|u\|_{H^{\gamma_1}_{p_1}}.
\end{equation}

\item \label{pointwise_multiplier}
Let $u\in H_p^\gamma$. Then, we have
\begin{equation*}
\| au \|_{H_p^\gamma} \leq N(d,\gamma,p)\| a \|_{B^{|\gamma|+\alpha}}\| u \|_{H_p^\gamma}\quad\text{and}\quad\| bu \|_{H_p^\gamma(\ell_2)} \leq N(d,\gamma,p)\| b \|_{B^{|\gamma|+\alpha}(\ell_2)}\| u \|_{H_p^\gamma},
\end{equation*}
where $B^{|\gamma|+\alpha}$ and $B^{|\gamma|+\alpha}(\ell_2)$ are introduced in Definition \ref{def_pointwise_multiplier}.
\end{enumerate}
\end{lemma}
\begin{proof}
These are well-known results. For example, for \eqref{dense_subset_bessel_potential}-\eqref{multi_ineq}, see Theorem 13.3.7 (i), Theorem 13.8.1, Theorem 13.3.10, Corollary 13.3.9, Theorem 13.3.7 (ii), and Exercise 13.3.20 of \cite{krylov2008lectures}, respectively. To obtain \eqref{eq:multi_ineq 2}, apply Young's inequality to \eqref{eq:multi_ineq 1}. For \eqref{pointwise_multiplier}, see  \cite[Lemma 5.2]{kry1999analytic}.
\end{proof}

\vspace{2mm}

Next we introduce stochastic Banach spaces. Let $(\Omega, \cF, P)$ be a complete probability space with a filtration $\{\cF_t\}_{t\geq0}$ satisfying the usual conditions. A set $\cP$ denotes the predictable $\sigma$-field related to $\cF_t$.

We introduce stochastic Banach spaces $\bH_p^{\gamma}(\tau)$ and $\cH_p^{\gamma}(\tau)$. For more detail, see \cite[Section 3]{kry1999analytic}.

\begin{definition}[Stochastic Banach spaces]
\label{definition_of_function_spaces}

Let $\tau\leq T$ be a bounded stopping time, $p>1$ and $\gamma\in\bR$. For $\opar0,\tau\cbrk:=\{ (\omega,t):0<t\leq \tau(\omega) \},$ define
\begin{gather*}
\mathbb{H}_{p}^{\gamma}(\tau) := L_p\left(\opar0,\tau\cbrk, \mathcal{P}, dP \times dt ; H_{p}^\gamma\right),\quad \mathbb{H}_{p}^{\gamma}(\tau,\ell_2) := L_p\left(\opar0,\tau\cbrk,\mathcal{P}, dP \times dt;H_{p}^\gamma(\ell_2)\right).
\end{gather*}
When $\gamma = 0$, we write $\mathbb{L}_{p}(\tau): = \bH_p^0(\tau)$ and $\mathbb{L}_{p}(\tau,\ell_2): = \bH_p^0(\tau,\ell_2)$.
For the initial data, set
$$U_{p}^{\gamma} :=  L_p\left(\Omega,\cF_0, dP ; H_{p}^{\gamma-2/p}\right).
$$
The norm of each space is defined as in \eqref{def_norm}. 
\end{definition}

\begin{definition}[Solution spaces] 
\label{definition_of_sol_space}
Let $\tau\leq T$ be a bounded stopping time and $p\geq2$. 

\begin{enumerate}[(i)]
\item 
For $u\in \bH_p^{\gamma+2}(\tau)$, we write $u\in\cH^{\gamma+2}_p(\tau)$ if there exist $u_0\in U_{p}^{\gamma+2}$ and  $(f,g)\in
\bH_{p}^{\gamma}(\tau)\times\bH_{p}^{\gamma+1}(\tau,\ell_2)$ such that
\begin{equation*}
du = fdt+\sum_{k=1}^{\infty} g^k dw_t^k,\quad   t\in (0, \tau]\,; \quad u(0,\cdot) = u_0
\end{equation*}
in the sense of distributions. In other words, for any $\phi\in \cS$, the equality
\begin{equation} \label{def_of_sol}
(u(t,\cdot),\phi) = (u_0,\phi) + \int_0^t(f(s,\cdot),\phi)ds + \sum_{k=1}^{\infty} \int_0^t(g^k(s,\cdot),\phi)dw_s^k
\end{equation}
holds for all $t\in [0,\tau]$ almost surely. Here, $\left\{w_t^k:k\in \bN\right\}$ is a set of one-dimensional independent Wiener processes. 

\item

The norm of the function space $\cH_{p}^{\gamma+2}(\tau)$ is defined as
\begin{equation}
\label{def_of_sol_norm}
\| u \|_{\cH_{p}^{\gamma+2}(\tau)} :=  \| u \|_{\mathbb{H}_{p}^{\gamma+2}(\tau)} + \| f \|_{\mathbb{H}_{p}^{\gamma}(\tau)} + \| g \|_{\mathbb{H}_{p}^{\gamma+1}(\tau,l_2)} + \| u_0 \|_{U_{p}^{\gamma+2}}.
\end{equation}

\item \label{def_of_local_sol_space}
For a stopping time $\tau \in [0,\infty]$, we write $u \in \cH_{p,loc}^{\gamma+2}(\tau)$ if there exists a sequence of bounded stopping times $\{ \tau_n : n\in\bN \}$ such that $\tau_n\uparrow \tau$ (a.s.) as $n\to\infty$ and $u\in \cH_{p}^{\gamma+2}(\tau_n)$ for each $n$. We write $u = v$ in $\cH_{p,loc}^{\gamma+2}(\tau)$ if there exists a sequence of bounded stopping times $\{ \tau_n : n\in\bN \}$ such that $\tau_n\uparrow\tau$ (a.s.) as $n\to\infty$ and $u = v$ in $\cH_{p}^{\gamma+2}(\tau_n)$ for each $n$. We omit $\tau$ if $\tau = \infty$. In other words,  $\cH_{p,loc}^{\gamma+2}=\cH_{p,loc}^{\gamma+2}(\infty)$.

\item If $\gamma+2 = 0$, we use $\cL$ instead of $\cH$. For example, $\cL_p(\tau) := \cH_p^0(\tau)$.

\end{enumerate}

\end{definition}

\begin{remark}
Let $p\geq 2$ and $\gamma\in \bR$. For any $g\in \bH^{\gamma+1}_{p}(\tau,\ell_2)$,  the series of stochastic integrals in \eqref{def_of_sol} converges uniformly in $t$ in probability on $[0,\tau \wedge T]$ for any $T$. Therefore, $(u(t,\cdot),\phi)$ is continuous in $t$ (See, e.g., \cite[Remark 3.2]{kry1999analytic}).
\end{remark}

Now, we introduce the properties of the solution space $\cH_p^{\gamma+2}(\tau)$.

\begin{theorem} 
\label{embedding}
Let $\tau\leq T$ be a bounded stopping time.
\begin{enumerate}[(i)]

\item \label{completeness}
For any $p\geq2$, $\gamma\in\bR$, $\cH_p^{\gamma+2}(\tau)$ is a Banach space with the norm $\| \cdot \|_{\cH_p^{\gamma+2}(\tau)}$.

\item \label{large-p-embedding} 
If $p>2$, $\gamma\in\bR$, and $1/p < \alpha < \beta < 1/2$, then for any  $u\in\cH_{p}^{\gamma+2}(\tau)$, we have $u\in C^{\alpha-1/p}\left([0,\tau];H_{p}^{\gamma+2-2\beta}\right)$ (a.s.) and
\begin{equation} 
\label{solution_embedding}
\mathbb{E}| u |^p_{C^{\alpha-1/p}\left([0,\tau];H_{p}^{\gamma + 2 - 2\beta} \right)} \leq N(d,p,\gamma,\alpha,\beta,T)\| u \|^p_{\cH_{p}^{\gamma+2}(\tau)}.
\end{equation}

\item \label{gronwall_type_ineq} 
Let $p > 2$, $\gamma\in\bR$, and $u\in\cH_p^{\gamma+2}(\tau)$. If there exist $N_0,N_1\in(0,\infty),$ and $\gamma_0 \in (0, \gamma)$ such that
\begin{equation}
\label{condition of gronwall type ineq}
\| u \|_{\cH_p^{\gamma+2}(\tau\wedge t)}^p \leq N_0 + N_1 \| u \|_{\bH_p^{\gamma_0+2}(\tau\wedge t)}^p
\end{equation}
for all $t\in(0,T)$, then we have
\begin{equation} \label{modified_Gronwall}
\| u \|_{\cH_p^{\gamma+2}(\tau\wedge T)}^p \leq N_0N,
\end{equation}
where $N =  N(N_1,d,p,\gamma,T)$.

\end{enumerate}
\end{theorem}
\begin{proof}
For \eqref{completeness} and \eqref{large-p-embedding}, see Theorems 3.7 and 7.2 of \cite{kry1999analytic}, respectively. Now we show \eqref{gronwall_type_ineq}. Given $\gamma_0+2 < \gamma+2$, by Lemma \ref{prop_of_bessel_space} \eqref{norm_bounded}, we have
\begin{equation}
\label{applying lemma 2.4.4}
\| u \|_{\bH^{\gamma_0+2}_p(\tau\wedge t)}^p \leq \| u \|_{\bH^{\gamma+2-\ep_0}_p(\tau\wedge t)}^p 
\end{equation}
for small $\ep_0>0$. Thus, by applying \eqref{applying lemma 2.4.4}, \eqref{eq:multi_ineq 2}, and \eqref{def_of_sol_norm} to \eqref{condition of gronwall type ineq}, we obtain
\begin{equation*}
\begin{aligned}
\|u\|_{\cH_p^{\gamma+2}(\tau\wedge t)}^p 
&\leq N_0 +  N_1 \| u  \|_{\bH^{\gamma_0+2}_p(\tau\wedge t)}^p \\
&\leq N_0 +  N_1 \| u  \|_{\bH^{\gamma+2-\ep_0}_p(\tau\wedge t)}^p \\
&\leq N_0 +  \frac{1}{2} \bE\int_0^{\tau\wedge t} \| u(s,\cdot) \|_{H^{\gamma+2}_p}^p ds + N \bE\int_0^{\tau\wedge t} \| u(s,\cdot) \|_{H^{\gamma}_p}^p ds\\
&\leq N_0 + \frac{1}{2}  \| u \|_{\cH^{\gamma+2}_p(\tau\wedge t)}^p + N \int_0^{t} \bE \sup_{r\leq \tau\wedge s}\| u(r,\cdot) \|_{H^{\gamma}_p}^p ds, \\
\end{aligned}
\end{equation*}
where $N = N(N_1, d,p,\gamma)$.
By subtraction and \eqref{solution_embedding}, we derive
\begin{equation*}
\begin{aligned}
\|u\|_{\cH_p^{\gamma+2}(\tau\wedge t)}^p
&\leq 2N_0 + 2N \int_0^{t} \| u \|_{\cH^{\gamma+2}_p(\tau\wedge  s)}^p ds,
\end{aligned}
\end{equation*}
where $N = N(N_1,d,p,\gamma)$. Using Gr\"onwall's inequality, we obtain \eqref{modified_Gronwall}. The theorem is proved.
\end{proof}

\begin{corollary} 
\label{embedding_corollary}

Let $\tau\leq T$ be a bounded stopping time.
Suppose $\kappa\in [0,1)$, $p\in(2,\infty)$, and $\alpha,\beta\in(0,\infty)$ satisfy
\begin{equation} \label{condition_for_alpha_beta}
\frac{1}{p} < \alpha < \beta <  \frac{1}{2}\left(1-\kappa-\frac{d}{p}\right).
\end{equation}
Then, we have
\begin{equation} \label{sol_embedding}
\mathbb{E} \|u\|^p_{C^{\alpha-1/p}([0,\tau];C^{1-\kappa-2\beta-d/p} )}\leq N\|u\|_{\cH^{1-\kappa}_{p} (\tau)}^p,
\end{equation}
where $N = N(d,p,\alpha,\beta,\kappa,T)$.

\end{corollary}
\begin{proof}

Set $\gamma = 1-\kappa-2\beta$. Then, by Lemma \ref{prop_of_bessel_space} \eqref{sobolev-embedding},  we obtain
\begin{equation} 
\label{holder_lp_embedding}
\begin{aligned}
\|u(t,\cdot)\|_{C^{1-\kappa-2\beta-d/p}}&\leq N\| u(t,\cdot) \|_{H_{p}^{1-\kappa-2\beta}}
\end{aligned}
\end{equation}
for all $t \in [0,\tau]$ almost surely. By \eqref{holder_lp_embedding} and \eqref{solution_embedding}, we have \eqref{sol_embedding}. The corollary is proved.

\end{proof}

\vspace{2mm}

\section{Main results} 
\label{sec:main_results}

In this section, we consider the equation of type
\begin{equation} 
\label{burger's_eq_white_noise_in_time}
\begin{aligned}
du &= \left(a^{ij}u_{x^ix^j} + b^i u_{x^i} + cu + \bar{b}^i(t,\bar x^i)|u|^{\lambda}u_{x^i}\right) dt + \sigma^k(t,x,u)dw_t^k,
\end{aligned}
\end{equation}
with nonnegative initial data $u(0,\cdot) = u_0(\cdot)$. Here, $(t,x)\in(0,\infty)\times\bR^d$, $\lambda\in(0,\infty)$ and $\bar x^i = (x^1,\dots,x^{i-1},x^{i+1},\dots,x^d)$. Depending on the types of diffusion coefficient $\sigma^k(u)$, we separate the problem into two cases. In Section \ref{subsec:Lipschitz_diffusion_case}, we assume the diffusion coefficient $\sigma^k(u)$ has a Lipschitz continuity and linear growth in $u$. On the other hand, in Section \ref{subsec:superlinear_diffusion_case}, we consider the case in which $\sigma^k(u)$ is super-linear in $u$. In each subsection, we provide the uniqueness, existence, and $L_p$-regularity of the solution; see Theorems \ref{theorem_burgers_eq_white_noise_in_time_large_lambda} and \ref{theorem_burgers_eq_white_noise_in_time_small_lambda}. In addition, we obtain the H\"older regularities of the solution; see Corollaries \ref{maximal holder regularity1} and \ref{maximal holder regularity2}.

To construct a global solution from the local solutions, the $L_q$ bound of solution $u$ with some $q\geq1$ is required. Precisely, if $\sigma^k(u)$ has a Lipschitz continuity and linear growth in $u$, the $\bL_{p(1+\lambda)}(\tau)$ bound of the solution $u$ is used, where $\tau$ is a bounded stopping time. On the other hand, if $\sigma^k(u)$ is super-linear in $u$ ($\sigma^k(u) = \mu^k |u|^{1+\lambda_0}$, where $\mu^k$ satisfies Assumption \ref{stochastic_part_assumption_on_coeffi_white_noise_in_time_small_lambda}), a uniform $L_1(\bR^d)$ bound of the solution $u$ is employed. This difference is due to the way we control the nonlinear diffusion coefficient $\sigma^k(u)$. When $\sigma^k(u)$ has a Lipschitz continuity and linear growth in $u$, It\^o's formula and the fundamental theorem of calculus provide $\bL_{p(1+\lambda)}(\tau)$ bound of the solution; see Lemma \ref{Lq_bound_infinite_noise}. However, if $\sigma^k(u)$ is super-linear in $u$, we use the conservation law to achieve a uniform $L_1$ bound of solution $u$. Indeed, if we apply It\^o's formula to obtain the $L_q$ bound for some $q>1$, the diffusion coefficient $\sigma^k(u) = \mu^k|u|^{1+\lambda_0}$ causes imbalanced exponents and prevents an $L_q$ $(q>1)$ estimate of the solution $u$. 
 For more details, see Remarks \ref{bar_b_i_indep_of_x_i_1}, \ref{remark: discussion of regularity in super-linear case}, and \ref{bar_b_i_indep_of_x_i_2}.

It also should be mentioned that different assumptions regarding the diffusion coefficient $\sigma^k(u)$ lead to different regularity results, dependency results, and conditions on parameters. If $\sigma^k(u)$ has a Lipschitz continuity and linear growth in $u$, $\lambda$ can be any positive real number; $\lambda\in(0,\infty)$. However, if $\sigma^k(u) = \mu^k |u|^{1+\lambda_0}$ is super-linear in $u$, we restrict the range of $\lambda$ to $\lambda\in(0,1/d)$. This difference comes from the fact that the nonlinear term $\sigma^k(u)$ affects the $L_q$ estimate of the solution. 
Additional explanations are included in Remarks \ref{remark:condition_on_lambda_1}, \ref{remark: discussion of regularity in super-linear case}, \ref{remark:condition_on_lambda_2}, and \ref{remark:condition_on_lambda_3}.

\vspace{2mm}


\subsection{ The first case: \texorpdfstring{$\lambda\in(0,\infty)$}{Lg} and the Lipschitz diffusion coefficient \texorpdfstring{$\sigma(u)$}{Lg}}
\label{subsec:Lipschitz_diffusion_case}

In this section, we study equation \eqref{burger's_eq_white_noise_in_time}
with $\lambda\in (0,\infty)$ and the Lipschitz function $\sigma(u)$. Our assumptions regarding the coefficients $a^{ij},b^i,c,\bar b^i$, and $\sigma^k(u)$ are described in Assumptions \ref{assumptions_on_coefficients_deterministic_and_linear_part}, \ref{assumptions_on_coefficients_deterministic_and_nonlinear_part}, and \ref{assumptions_on_coefficients_stochastic_part}. It should be noted that the coefficient $\bar b^i$ is assumed to be independent of $x^i$, and it is a sufficient condition for the existence of a global solution; see Remark \ref{bar_b_i_indep_of_x_i_1}. In Theorem \ref{theorem_burgers_eq_white_noise_in_time_large_lambda}, we provide the uniqueness, existence, and $L_p$ regularity of the solution. Furthermore, by applying the H\"older embedding theorem (Corollary \ref{embedding_corollary}), we obtain H\"older regularity of the solution. To achieve the H\"older regularity of the solution, we prove the uniqueness of the solution in $p$; see Theorem \ref{uniqueness_in_p_1}. The H\"older regularity of the solution is provided in Corollary \ref{maximal holder regularity1}.

\vspace{1mm}

Below, we introduce assumptions on coefficients.

\begin{assumption} 
\label{assumptions_on_coefficients_deterministic_and_linear_part}
\begin{enumerate}[(i)]

\item 
The coefficients $a^{ij} = a^{ij}(t,x)$, $b^i = b^i(t,x)$, and $c = c(t,x)$ are $\cP\times\cB(\bR^d)$-measurable.

\item There exists $K>0$ such that 
\begin{equation}
\label{ellipticity_of_leading_coefficients_white_noise_in_time} 
K^{-1}|{\xi}|^2 \leq \sum_{i,j}a^{ij}\xi^i\xi^j \leq  K|{\xi}|^2\quad \text{for all}\quad (\omega,t,x)\in\Omega\times[0,\infty)\times\bR^d, \quad{\xi} = (\xi^1,\dots,\xi^d)\in \bR^d, 
\end{equation}
and
\begin{equation} 
\label{boundedness_of_deterministic_coefficients_white_noise_in_time} 
\sum_{i,j}\left| a^{ij}(t,\cdot) \right|_{C^{2}(\bR^d)} + \sum_{i}\left| b^{i}(t,\cdot) \right|_{C^{2}(\bR^d)} + |c(t,\cdot)|_{C^{2}(\bR^d)} \leq K
\end{equation}
for all $(\omega,t)\in\Omega\times[0,\infty)$.
\end{enumerate}
\end{assumption}

\begin{assumption}
\label{assumptions_on_coefficients_deterministic_and_nonlinear_part}
\begin{enumerate}[(i)]
\item 
The coefficient $\bar{b}^i = \bar{b}^i(t,\bar x^i) = \bar{b}^i(t,x^1,\dots,x^{i-1},x^{i-1},\dots,x^d)$ is $\cP\times\cB(\bR^{d-1})$-measurable.

\item There exists $K>0$ such that
\begin{equation*}
\sum_{i}\left| \bar{b}^i(t,\cdot) \right|_{C^2(\bR^{d-1})} < K
\end{equation*}
for all $(\omega,t)\in\Omega\times[0,\infty)$.
\end{enumerate}
\end{assumption}

\begin{remark}
If $d = 1$, the coefficient $\bar b$ is assumed to be independent of $x$.
\end{remark}

\begin{assumption} 
\label{assumptions_on_coefficients_stochastic_part}
\begin{enumerate}[(i)]

\item 
The coefficient $\sigma^k(t,x,u)$ is $\cP\times\cB(\bR^d)\times\cB(\bR)$-measurable.

\item There exists $K>0$ such that for $\omega\in\Omega, t>0, x,u,v\in\bR$,
\begin{equation}
\label{boundedness_of_stochastic_coefficients_white_noise_in_time_large_lambda}
\left| \sigma(t,x,u) - \sigma(t,x,v) \right|_{\ell_2} \leq K |u-v|\quad\text{and}\quad\left| \sigma(t,x,u) \right|_{\ell_2} \leq K |u|,
\end{equation}
where $\sigma(t,x,u) = \left( \sigma^1(t,x,u),\sigma^2(t,x,u),\dots \right)$.
\end{enumerate}
\end{assumption}

Below we provide the main results of this section.

\begin{theorem} \label{theorem_burgers_eq_white_noise_in_time_large_lambda}

Let $\lambda \in (0,\infty)$ and $p>d+2$. Suppose Assumptions \ref{assumptions_on_coefficients_deterministic_and_linear_part}, \ref{assumptions_on_coefficients_deterministic_and_nonlinear_part}, and \ref{assumptions_on_coefficients_stochastic_part} hold. For the nonnegative initial data $u_0\in U_p^{1}\cap L_{p(1+\lambda)}(\Omega;L_{p(1+\lambda)})$, equation \eqref{burger's_eq_white_noise_in_time}
has a unique nonnegative solution $u$ in $\cH_{p,loc}^{1}$. Furthermore, if $\alpha$ and $\beta$ satisfy \eqref{condition_for_alpha_beta} with $\kappa=0$,
then for any $T<\infty$, we have (a.s.)
\begin{equation}
\label{holder_regularity_main_theorem}
\|u\|^p_{C^{\alpha-1/p}([0,t];C^{1-2\beta-d/p}(\bR^d) )} < \infty.
\end{equation}

\end{theorem}
\begin{proof}
See {\bf{Proof of Theorem \ref{theorem_burgers_eq_white_noise_in_time_large_lambda}}} in Section \ref{proof_of_theorem_white_noise_in_time_large_lambda}.
\end{proof}

\begin{remark}
\label{bar_b_i_indep_of_x_i_1}


Notice that the coefficient $\bar b^i$ is assumed to be independent of $x^i$, and it is a sufficient condition for the existence of a global solution. Indeed, to prove the existence of a solution, we need to show that $\|u\|_{\cH_p^1(\tau)}$ is bounded for any stopping time $\tau\leq T$, where $T\in(0,\infty)$. Then, it is proven that $\| u \|_{\cH_p^1(\tau)}$ is controlled by a constant times of $\| u_0 \|_{U_p^1}+\left\| u  \right\|_{\bL_{p(1+\lambda)}(\tau)}$, and the constant is independent of $\tau$ (see Theorem \ref{theorem_nonlinear_case} and \eqref{nonlinear term estimate}). Furthermore, one can observe that $\|u\|_{\bL_{p(1+\lambda)}(\tau)}$ is bounded by $\| u_0 \|_{L_{p(1+\lambda)}(\Omega;L_{p(1+\lambda)})}$ since $\bar b^i$ is independent of $x^i$ (Lemma \ref{Lq_bound_infinite_noise}). In detail, we employ the fundamental theorem of calculus to remove the terms related to $\bar b^i \left( |u|^{1+\lambda} \right)_{x^i}$. In \eqref{Lq_bound_nonlinear_noise_zero}, the coefficient $\bar b^i$ can be out of the integral with respect to $x^i$ since $\bar b^i$ is independent of $x^i$, and thus, the terms related to the nonlinear term vanish by integrating $\frac{\partial}{\partial x^i}(u(s,x))^{q+\lambda-1}$ with respect to $x^i$.

\end{remark}

\begin{remark}
\label{existence_of_alpha_beta}
Since $p>d+2$, we can choose $\alpha$ and $\beta$ satisfying \eqref{condition_for_alpha_beta} with $\kappa = 0$.
\end{remark}

\begin{remark}
The initial data $u_0$ should satisfy the summability condition \linebreak $u_0\in L_{p(1+\lambda)}(\Omega;L_{p(1+\lambda)})$. This assumption is used to calculate the $\bL_{p(1+\lambda)}(\tau)$ estimate of the solution.
\end{remark}

\begin{remark}
\label{remark:condition_on_lambda_1}

It should be noted that any positive $\lambda>0$ is admitted in Theorem \ref{theorem_burgers_eq_white_noise_in_time_large_lambda} and that the regularity of the solution is independent of $\lambda$. Indeed, the nonlinear term $\bar b^i \left(|u|^{1+\lambda}\right)_{x^i}$ is removed by the fundamental theorem of calculus while we achieve the $\bL_{p(1+\lambda)}(\tau)$ estimate of the solution; see Remark \ref{bar_b_i_indep_of_x_i_1}.

\end{remark}



\begin{remark}\label{remark:Holder_regularity}
The H\"older regularity of the solution introduced in Theorem \ref{theorem_burgers_eq_white_noise_in_time_large_lambda} depends on $\alpha$ and $\beta$. For example, for $\ep\in\left(0,1-\frac{d+2}{p}\right)$, let $\alpha = \frac{1}{p}+\frac{\ep}{4}$ and $\beta = \frac{1}{p}+\frac{\ep}{2}$. Then, we have
\begin{equation*}
\sup_{t\leq T}|u(t,\cdot)|_{C^{1-\frac{d+2}{p}-\ep}(\bR^d)} < \infty,
\end{equation*}
almost surely. Similarly, for $\ep\in \left(0,\frac{1}{2}-\frac{d+2}{2p}\right)$, set $\alpha = \frac{1}{2}\left(1-\frac{d}{p}\right)-\ep$ and $\beta = \frac{1}{2}\left(1-\frac{d}{p}\right)-\frac{\ep}{2}$. Then, we have
\begin{equation*}
\sup_{x\in\bR^d}|u(\cdot,x)|_{C^{\frac{1}{2}-\frac{d+2}{2p}-\ep}([0,T])}  < \infty
\end{equation*}
almost surely.

\end{remark}

To obtain the H\"older regularity of the solution, $p$ should be sufficiently large. Therefore, we prove the uniqueness of the solution in $p$.

\begin{theorem}
\label{uniqueness_in_p_1}
Assume that all the conditions of Theorem \ref{theorem_burgers_eq_white_noise_in_time_large_lambda} hold. Let $u\in \cH_{p,loc}^{1}$ be the solution to equation \eqref{burger's_eq_white_noise_in_time} introduced in Theorem \ref{theorem_burgers_eq_white_noise_in_time_large_lambda}. If $r>p$ and $u_0\in U_{r}^{1}\cap L_{r(1+\lambda)}(\Omega;L_{r(1+\lambda)})$, then $u\in \cH_{r,loc}^{1}$.
\end{theorem}
\begin{proof}
See {\bf{Proof of Theorem \ref{uniqueness_in_p_1}}} in Section \ref{proof_of_theorem_white_noise_in_time_large_lambda}.
\end{proof}

By combining Theorems \ref{theorem_burgers_eq_white_noise_in_time_large_lambda} and \ref{uniqueness_in_p_1}, we derive the  H\"older regularity of the solution.

\begin{corollary}
\label{maximal holder regularity1}
Suppose $u_0\in U_p^1$ for all $p>2$. Then, for small $\ep>0$, we have
\begin{equation}
\label{holder_regularity_large_lambda}
\sup_{t\leq T}|u(t,\cdot)|_{C^{1-\ep}(\bR^d)} + \sup_{x\in\bR^d}|u(\cdot,x)|_{C^{1/2-\ep}([0,T])}  < \infty\quad\text{for all} \quad T<\infty,
\end{equation}
almost surely.

\end{corollary}
\begin{proof}



For each $p > d + 2$, according to Theorem \ref{theorem_burgers_eq_white_noise_in_time_large_lambda}, there exists a unique solution $u = u_p\in \cH_p^1(T)$ to equation \eqref{burger's_eq_white_noise_in_time} for any $T<\infty$. Since $u_0 \in U_p^1$ for all $p > 2$, by Theorem \ref{uniqueness_in_p_1}, all the solutions $u$ coincide. In Corollary \ref{embedding_corollary}, for sufficiently small $\ep>0$, consider $\kappa = 0$, $p = \frac{d+2}{\ep}$, $\alpha = \frac{1}{p} + \frac{\ep}{8}$, and $\beta = \frac{1}{p} + \frac{\ep}{4}$. Then, 
\begin{equation*}
\sup_{t\leq T} |u(t,\cdot)|_{C^{1-\ep}(\bR^d)}^p \leq |u|_{C^{\alpha-\frac{1}{p}}\left([0,T];C^{1-2\beta-\frac{d}{p}}(\bR^d)\right)}^p<\infty
\end{equation*}
for any $T<\infty$ almost surely. This implies the maximum regularity in the direction of $x$. On the other hand, in Corollary \ref{embedding_corollary}, for sufficiently small $\ep>0$, consider $\kappa = 0$, $p = \frac{d+2}{\ep}$, $\alpha = \frac{1}{2}\left( 1-\frac{d}{p} \right) - \frac{\ep}{2}$, and $\beta = \frac{1}{2}\left( 1 - \frac{d}{p} \right)- \frac{\ep}{4}$. Then, 
\begin{equation*}
\sup_{x\in \bR^d} |u(\cdot,x)|_{C^{\frac{1}{2} - \ep}([0,T])}\leq |u|_{C^{\alpha-\frac{1}{p}}\left([0,T];C^{1-2\beta-\frac{d}{p}}(\bR^d)\right)}^p<\infty
\end{equation*}
for any $T<\infty$ almost surely. This yields maximum regularity in the direction of $t$. In conclusion, we have \eqref{holder_regularity_large_lambda}. The corollary is proved.

\end{proof}

\vspace{2mm}


\subsection{The second case: \texorpdfstring{$\lambda\in(0,1/d)$}{Lg} and the super-linear diffusion coefficient}
\label{subsec:superlinear_diffusion_case}

Consider a semilinear stochastic partial differential equation with super-linear diffusion coefficient
\begin{equation} 
\label{burger's_eq_white_noise_in_time_small}
du = \left(a^{ij}u_{x^ix^j} + b^i u_{x^i} + cu + \bar{b}^i|u|^{\lambda} u_{x^i}\right) dt + \mu^k |u|^{1+\lambda_0}dw_t^k,\quad (t,x)\in(0,\infty)\times\bR^d
\end{equation}
with $u(0,\cdot) = u_0(\cdot)$, where $\lambda,\lambda_0\in(0,1/d)$. 
As in Section \ref{subsec:Lipschitz_diffusion_case}, the coefficients $a^{ij},b^i,c$, and $\bar b^i$ are assumed to satisfy Assumptions \ref{assumptions_on_coefficients_deterministic_and_linear_part} and \ref{assumptions_on_coefficients_deterministic_and_nonlinear_part}. Note that the coefficient $\bar b^i$ is independent of $x^i$ and it is a sufficient condition for the global existence of a solution; see Remark \ref{bar_b_i_indep_of_x_i_2}. In terms of the conditions of $\mu^k$, we employ Assumption \ref{stochastic_part_assumption_on_coeffi_white_noise_in_time_small_lambda}. 

In Theorem \ref{theorem_burgers_eq_white_noise_in_time_small_lambda}, we obtain the uniqueness, existence, and $L_p$ regularity of the solution. Furthermore, the H\"older regularity of the solution is achieved by Corollary \ref{embedding_corollary}. To see the H\"older regularity of the solution, the solution should be unique in $p$; see Theorem \ref{uniqueness_in_p_2}. In Corollary \ref{maximal holder regularity2}, the  H\"older regularity of the solution is achieved. 

\vspace{1mm}



\begin{assumption} 
\label{stochastic_part_assumption_on_coeffi_white_noise_in_time_small_lambda}
\begin{enumerate}[(i)]

\item 
The coefficient $\mu^k(t,x)$ is $\cP\times\cB(\bR^d)$-measurable.

\item There exists $K>0$ such that for $\omega\in\Omega, t>0, x\in\bR^d$,
\begin{equation}
\label{boundedness_of_stochastic_coefficients_white_noise_in_time_small_lambda}
\left( \sum_{k}|\mu^k(t,\cdot)|^2 \right)_{C(\bR^d)} \leq K.
\end{equation}
\end{enumerate}
\end{assumption}

Now, we present the main results.

\begin{theorem}
\label{theorem_burgers_eq_white_noise_in_time_small_lambda}
Let $\lambda, \lambda_0  \in (0,1/d)$, $1>\kappa> (\lambda d) \vee (\lambda_0 d) $, and $p > \frac{d+2}{1-\kappa}$. Suppose Assumptions \ref{assumptions_on_coefficients_deterministic_and_linear_part}, \ref{assumptions_on_coefficients_deterministic_and_nonlinear_part}, and \ref{stochastic_part_assumption_on_coeffi_white_noise_in_time_small_lambda} hold. For the nonnegative initial data $u_0\in U_p^{1-\kappa}\cap L_{1}(\Omega;L_1)$, equation \eqref{burger's_eq_white_noise_in_time_small} has a unique nonnegative solution $u$ in $\cH_{p,loc}^{1-\kappa}$. Furthermore, for any $\alpha,\beta$ satisfying \eqref{condition_for_alpha_beta},
we have (a.s.)
\begin{equation}
\label{holder_regularity_main_theorem_small_lambda}
\| u \|_{C^{\alpha-\frac{1}{p}}\left( [0,T];C^{1-\kappa-2\beta-d.p}\left( \bR^d \right) \right)} < \infty.
\end{equation}

\end{theorem}
\begin{proof}
See {\bf{Proof of Theorem \ref{theorem_burgers_eq_white_noise_in_time_small_lambda}}} in Section \ref{proof_of_theorem_white_noise_in_time_small_lambda}.
\end{proof}

\begin{remark}
\label{remark: discussion of regularity in super-linear case}



In the proof of Theorem \ref{theorem_burgers_eq_white_noise_in_time_small_lambda}, we employ different methods compared with Theorem \ref{theorem_burgers_eq_white_noise_in_time_large_lambda}.
To prove Theorem \ref{theorem_burgers_eq_white_noise_in_time_large_lambda}, we use It\^o's formula for the $\bL_{p(1+\lambda)}(\tau)$ bound of the solution. However, it is not easy to employ this method for Theorem \ref{theorem_burgers_eq_white_noise_in_time_small_lambda} since the super-linear diffusion coefficient $\mu^k |u|^{1+\lambda_0}$ causes imbalanced exponents. More precisely, if we apply It\^o's formula to equation \eqref{burger's_eq_white_noise_in_time_small} to obtain the $L_q$ bound for some $q>1$, we have
\begin{equation}
\label{when we apply ito formula with super-linear diffusion coefficient}
\begin{aligned}
&\bE e^{-M\tau_{m,n}\wedge t}\int_{\bR^d} |u(\tau_{m,n} \wedge t,x)|^q dx \\
&\quad \leq  N\bE \int_{\bR^d} |u_0(x)|^q dx + N\bE \int_0^{\tau_{m,n}\wedge t}\int_{\bR^d} \left|u(s,x)\right|^{q+\lambda} e^{-Ms}dxds
\end{aligned}
\end{equation}
instead of \eqref{before final estimate}. Notice that the last term of \eqref{when we apply ito formula with super-linear diffusion coefficient} cannot be controlled easily with the usual integral inequality. Therefore, we consider a uniform $L_1(\bR^d)$ bound of the solution by using the conservation law.

\end{remark}

\begin{remark}
\label{bar_b_i_indep_of_x_i_2} 

Since the coefficient $\bar b^i$ is independent of $x^i$ (Assumption \ref{assumptions_on_coefficients_deterministic_and_nonlinear_part}), we can construct a global solution. To be specific, we build a candidate of a global solution by pasting the local solutions, and we employ the uniform $L_1(\bR^d)$ bound of the local solutions to prove the candidate is a global solution. To obtain the $L_1(\bR^d)$ norm of the local solutions uniformly bounded in $t$, we prove that the $L_1(\bR^d)$ norm of the local solution is a local martingale. In the $L_1(\bR^d)$ estimate, the nonlinear component $\bar b^i |u|^\lambda u_{x^i}$ vanishes using the fundamental theorem of calculus since the nonlinear term is interpreted as
$$ \bar b^i |u|^\lambda u_{x^i} = \frac{1}{1+\lambda}\bar b^i \left( |u|^{1+\lambda} \right)_{x^i}.
$$
For more detail, see Lemma \ref{L_1_bound}.

\end{remark}

\begin{remark}
Since $p>\frac{d+2}{1-\kappa}$, there exist $\alpha,\beta$ that satisfy \eqref{condition_for_alpha_beta}.

\end{remark}



\begin{remark}
\label{remark:condition_on_lambda_2}
The algebraic conditions $\lambda,\lambda_0\in(0,1/d)$ are sufficient conditions for the existence of a global solution. To demonstrate this assertion, consider
\begin{equation}
\label{motivation_equation}
dv = \left(a^{ij}v_{x^ix^j} + b^i v_{x^i} + cv + \bar{b}^i\left( \xi v \right)_{x^i}\right) dt + \mu^k\xi_0\,v\, dw_t^k,\quad t > 0\,; \quad v(0,\cdot) = u_0(\cdot),
\end{equation}
where $\xi$ and $\xi_0$ are $\cP\times\cB(\bR^d)$-measurable functions. Formally, we can regard \eqref{burger's_eq_white_noise_in_time_small} as an equation type of \eqref{motivation_equation} by setting $\xi = |u|^{\lambda}$ and $\xi_0 = |u|^{\lambda_0}$.

In Lemma \ref{lemma: sol of eq with xi}, it emerges that if
\begin{equation}
\label{conditions_on_xi}
\xi\in L_\infty\left(\opar0,T\cbrk;L_s\right)\quad\text{and}\quad\xi_0\in L_\infty\left(\opar0,T\cbrk;L_{s_0}\right)\quad (s,s_0>d),
\end{equation} 
then for $\kappa\in(d/s \vee d/s_0, 1)$, there is a unique solution $v\in\cH_{p}^{1-\kappa}(T)$ to equation \eqref{motivation_equation}. Thus, when we consider \eqref{burger's_eq_white_noise_in_time_small} as \eqref{motivation_equation}, if we set $\lambda = 1/s$ and $\lambda_0 = 1/s_0$, then $\xi = |u|^\lambda$ and $\xi_0 = |u|^{\lambda_0}$ satisfy \eqref{conditions_on_xi} since the uniform $L_1$ bound of $u$ is provided in Lemma \ref{L_1_bound}. Here, $\lambda$ and $\lambda_0$ should be smaller than $1/d$ because 
$1/\lambda = s > d$ and $1/\lambda_0 = s_0 > d$. For a more detailed calculation, see the introduction to Section \ref{proof_of_theorem_white_noise_in_time_small_lambda}.

\end{remark}

\begin{remark}
\label{remark:condition_on_lambda_3}

Observe that $\lambda$ and $\lambda_0$ affect the regularity of the solution. In addition, the $L_p$ regularity of the solution in Theorem \ref{theorem_burgers_eq_white_noise_in_time_small_lambda} is less than $1-(\lambda d)\vee(\lambda_0 d)$, whereas that of Theorem \ref{theorem_burgers_eq_white_noise_in_time_large_lambda} is $1$. Furthermore, if we let $\lambda_0\downarrow0$, the regularity of the solution in Theorem \ref{theorem_burgers_eq_white_noise_in_time_small_lambda} is $1-\lambda d$, and it does not coincide with that of Theorem \ref{theorem_burgers_eq_white_noise_in_time_large_lambda}. However, if we consider $\lambda\downarrow 0$, then the regularity of solution is $1-\lambda_0 d$, and, in principle, it agrees with the regularity result of the stochastic partial differential equation driven by colored noise with a super-linear diffusion coefficient; see \cite[Theorem 3.10]{choi2021regularity}. These relationships give an impression that the proof method is primarily focused on addressing the super-linearity of the diffusion coefficient.

To compare this case to others with a similar type of equation, we consider an equation driven by space-time white noise with a super-linear diffusion coefficient. As implied in \cite[Section 8.4]{kry1999analytic} and \cite{mueller2000critical}, the $L_q(\bR^d)$ bound of the solution with $q>1$ cannot be obtained if the diffusion coefficient is super-linear. Thus, it seems natural that the $L_q(\bR^d)$ bound of the solution with $q>1$ cannot be achieved even in the case of an equation driven by colored noise. Therefore, there is a restriction of $\lambda$; see Remark \ref{remark:condition_on_lambda_2}. Thus, we expect the regularity results of Theorem \ref{theorem_burgers_eq_white_noise_in_time_small_lambda} to be nearly optimal, and we plan to confirm this assertion in future research.

\end{remark}

\begin{remark}
The H\"older regularity of the solution introduced in Theorem \ref{theorem_burgers_eq_white_noise_in_time_small_lambda} depends on $\alpha$ and $\beta$. As in Remark \ref{remark:Holder_regularity}, 
we have
\begin{equation*}
\sup_{t\leq T}|u(t,\cdot)|_{C^{1-\kappa-\frac{d+2}{p}-\ep}(\bR^d)} +\sup_{x\in\bR^d}|u(\cdot,x)|_{C^{\frac{1}{2}-\frac{d+2}{2p}-\ep}([0,T])}  < \infty
\end{equation*}
almost surely.

\end{remark}

As in Section \ref{subsec:Lipschitz_diffusion_case}, we prove the uniqueness of the solution in $p$. It is employed to obtain the H\"older regularity of the solution.

\begin{theorem}
\label{uniqueness_in_p_2}
Assume that all the conditions of Theorem \ref{theorem_burgers_eq_white_noise_in_time_small_lambda} holds. Let $u\in \cH_{p,loc}^{1-\kappa}$ be the solution to equation \eqref{burger's_eq_white_noise_in_time_small}, which is introduced in Theorem \ref{theorem_burgers_eq_white_noise_in_time_small_lambda}. If $r>p$ and $u_0\in U_r^{1-\kappa}\cap L_{1}(\Omega;L_1)$, then $u\in \cH_{r,loc}^{1-\kappa}$.
\end{theorem}

\begin{corollary}
\label{maximal holder regularity2}
Suppose $u_0\in U_p^{1-(\lambda d)\vee (\lambda_0 d)}\cap L_1(\Omega;L_1)$ for all $p>2$. Then, for small $\ep>0$, we have
\begin{equation}
\label{holder_regularity_small_lambda}
\sup_{t\leq T}|u(t,\cdot)|_{C^{1-(\lambda d)\vee (\lambda_0 d)-\ep}(\bR^d)} + \sup_{x\in\bR^d}|u(\cdot,x)|_{C^{\frac{1-(\lambda d)\vee (\lambda_0 d)}{2}-\ep}([0,T])} < \infty \quad\text{for all}\quad T<\infty,
\end{equation}
almost surely.

\end{corollary}
\begin{proof}
Since the proof of the corollary is similar to that for of Corollary \ref{maximal holder regularity1}, we only identify the differences. We employ Theorems \ref{theorem_burgers_eq_white_noise_in_time_small_lambda} and \ref{uniqueness_in_p_2}, instead of Theorems \ref{theorem_burgers_eq_white_noise_in_time_large_lambda} and \ref{uniqueness_in_p_1}, respectively.

\end{proof}

\vspace{2mm}

\section{Proof of the first case: \texorpdfstring{$\lambda\in(0,\infty)$}{Lg} and the Lipschitz diffusion coefficient }
\label{proof_of_theorem_white_noise_in_time_large_lambda}

In this section, proof of Theorem \ref{theorem_burgers_eq_white_noise_in_time_large_lambda} is provided. In Theorem \ref{theorem_nonlinear_case}, we first recall the $L_p$ regularity results of semilinear stochastic partial differential equation
\begin{equation} \label{nonlinear_equation}
du = \left(a^{ij}u_{x^ix^j} + b^i u_{x^i} + cu + f(u)\right) dt + g^k(u) dw_t^k,\quad (t,x)\in(0,\infty)\times\bR^d\,; \quad u(0,\cdot) = u_0(\cdot),
\end{equation}
where $f$ and $g$ satisfy Assumption \ref{assumption_on_f_and_g}. Based on Theorem \ref{theorem_nonlinear_case}, by cutting-off nonlinear term $\bar b^i \left(|u|^{1+\lambda}\right)_{x^i}$ and defining a stopping time using the first hitting time, we obtain a local solution to equation \eqref{burger's_eq_white_noise_in_time}; see Lemmas \ref{cut_off_lemma_large_lambda} and \ref{local_existence}. To extend the solution's existence time, its $\bL_{p(1+\lambda)}$ bound is required. Therefore, we show that $\| u \|_{\bL_q(\tau)}$ ($q>p$) is bounded by a constant independent of $u$ and $\tau$; see Lemma \ref{Lq_bound_infinite_noise}. At the end of this section, proofs for Theorems \ref{theorem_burgers_eq_white_noise_in_time_large_lambda} and \ref{uniqueness_in_p_1} are suggested.

\vspace{2mm}

Below are the assumptions regarding $f$ and $g$.

\vspace{2mm}

\begin{assumption}[$\tau$] \label{assumption_on_f_and_g}

Let $\tau\leq T$ be a bounded stopping time.

\begin{enumerate}[(i)]

\item The functions $f(t,x,u)$ and $g^k(t,x,u)$ are $\cP\times\cB(\bR^d)\times\cB(\bR)$-measurable satisfying $f(t,x,0)\in \bH_{p}^{\gamma}(\tau)$ and $g(t,x,0) = \left( g^1(t,x,0),g^2(t,x,0),\dots \right) \in \bH_{p}^{\gamma+1}(\tau,\ell_2).$



\item
For any $\ep>0$, there exists a constant $N_\ep$ such that
\begin{equation}
\label{conditions_on_f_and_g}
\| f(u) - f(v) \|_{\bH_{p}^\gamma(\tau)} + \| g(u) - g(v) \|_{\bH_{p}^{\gamma+1}(\tau,\ell_2)} \leq \ep\| u-v \|_{\bH_p^{\gamma+2}(\tau)} + N_\ep \| u-v \|_{\bH_p^{\gamma+1}(\tau)}
\end{equation}
for any $u,v\in \bH_p^{\gamma+2}(\tau)$.

\end{enumerate}
\end{assumption}

The next result is the $L_p$-solvability of equation \eqref{nonlinear_equation}.

\begin{theorem} 
\label{theorem_nonlinear_case}
Let $\tau\leq T$ be a bounded stopping time, $\gamma\in[-2,-1]$, and $p\geq2$.  Suppose Assumptions \ref{ellipticity_of_leading_coefficients_white_noise_in_time} and \ref{assumption_on_f_and_g} ($\tau$) hold. Then, for any $u_0\in U_p^{\gamma+2}$, equation \eqref{nonlinear_equation} has a unique solution $u\in\cH_p^{\gamma+2}(\tau)$ such that
\begin{equation} 
\label{nonlinear_estimate}
\| u \|_{\cH_p^{\gamma+2}(\tau)} \leq N\left(\|f(0)\|_{\bH_p^{\gamma}(\tau)} + \|g(0)\|_{\bH_p^{\gamma+1}(\tau,\ell_2)} + \| u_0 \|_{U_p^{\gamma+2}}\right),
\end{equation}
where $N = N(d,p,\gamma,K,T,N_\ep)$. 
\end{theorem}
\begin{proof}
Since \cite[Theorem 5.1]{kry1999analytic} implies the results when $\tau = T$, we only consider the case in which $\tau \leq T$. First, we prove the existence of a solution. Set
$$
\bar{f}(t, u) := 1_{t\leq \tau} f(t,u)\quad\text{and}\quad \bar{g}(t, u) := 1_{t\leq \tau} g(t,u).
$$
Note that $\bar{f}(u)$ and $\bar{g}(u)$ satisfy \eqref{conditions_on_f_and_g} for any $u,v\in\bH_p^{\gamma+2}(T)$. Therefore, according to \cite[Theorem 5.1]{kry1999analytic}, there exists a unique solution $u\in \cH_{p}^{\gamma+2}(T)$ such that $u$ satisfies equation \eqref{nonlinear_equation} with $\bar f$ and $\bar{g}$ instead of $f$ and $g$, respectively. Since $\tau\leq T$, we have $u\in\cH_{p}^{\gamma+2}(\tau)$, and $u$ satisfies equation \eqref{nonlinear_equation}. Additionally, we obtain \eqref{nonlinear_estimate} with $f,g$.

Next, we demonstrate uniqueness. Suppose $u,v\in \cH_{p}^{\gamma+2}(\tau)$ are two solutions to equation \eqref{nonlinear_equation}. Then, according to \cite[Theorem 5.1]{kry1999analytic}, there exists a unique solution $\bar{v}\in \cH_{p}^{\gamma+2}(T)$ satisfying
\begin{equation}
\label{equation_in_proof_of_uniqueness}
d\bar{v}=\left(a^{ij}\bar{v}_{x^ix^j} + b^i\bar{v}_{x^i} + c\bar{v} + \bar{f}(v) \right)dt + \sum_{k=1}^{\infty} \bar{g}^k(v)  dw^k_t, \quad 0<t<T\, ; \quad \bar{v}(0,\cdot)=u_0.
\end{equation}
It should be noted that in \eqref{equation_in_proof_of_uniqueness}, $\bar{f}(v)$ and $\bar{g}(v)$ are used instead of $\bar{f}(\bar{v})$ and $\bar{g}(\bar{v})$, respectively. We set  $\tilde v:=v-\bar{v}$. Then, for a fixed $\omega\in\Omega$, we obtain
$$
d\tilde v=\left(a^{ij}\tilde v_{x^ix^j} + b^i\tilde v_{x^i} + c\tilde v \right)dt, \quad 0<t<\tau\, ; \quad \tilde v(0,\cdot)=0.
$$
By the deterministic version of \cite[Theorem 5.1]{kry1999analytic}, we have $\tilde v = 0$ for all $t < \tau$ almost surely, and this implies that $v(t,\cdot) = \bar{v}(t,\cdot)$ in $H_p^{\gamma+2}$ for all $t < \tau$ for each $\omega\in\Omega$. Thus, in equation \eqref{equation_in_proof_of_uniqueness}, we can replace $\bar f (v),\bar{g}(v)$ with $\bar f (\bar v),\bar{g}(\bar{v})$. Therefore, $\bar v\in \cH_p^{\gamma+2}(T)$ satisfies equation \eqref{nonlinear_equation} using $\bar f, \bar g$ instead of $f,g$, respectively. Similarly, there exists $\bar{u}\in \cH_p^{\gamma+2}(T)$ such that $\bar{u}$ satisfies the equation 
$$d\bar{u}=\left(a^{ij}\bar{u}_{x^ix^j} + b^i\bar{u}_{x^i} + c\bar{u} + \bar{f}(u) \right)dt + \sum_{k=1}^{\infty} \bar{g}^k(u)  dw^k_t, \quad 0<t<T\, ; \quad \bar{u}(0,\cdot)=u_0.
$$ 
and $u(t,\cdot) = \bar{u}(t,\cdot)$ in $H_p^{\gamma+2}$ for all $t\leq \tau$ almost surely. Therefore, $\bar{f}(u), \bar{g}^k(u)$ are replaced $\bar{f}(\bar u), \bar{g}^k(\bar u)$. Thus, by the uniqueness result in $\cH_{p}^{\gamma+2}(T)$, we obtain $\bar{u} = \bar{v}$ in $\cH_{p}^{\gamma+2}(T)$, which implies $u = v$ in $H_p^{\gamma+2}$ for almost every $(\omega,t)\in\opar0,\tau\cbrk$. Therefore, we have $u = v$ in $\cH_p^{\gamma+2}(\tau)$. The theorem is proved.
\end{proof}

\vspace{2mm}

The following two lemmas show the existence of a local solution to equation \eqref{burger's_eq_white_noise_in_time}. To obtain the local solution, we introduce a cut-off function. Consider a real-valued nonnegative function $h(z)\in C^1(\bR)$ such that $h(z) = 1$ on $|z|\leq1$ and $h(z) = 0$ on $|z|\geq2$. For $m\in\bN$, set 
\begin{equation}
\label{def of hm}
h_m(z) := h(z/m).
\end{equation}
By employing $((u_+)^{1+\lambda}h_m(u))_{x^i}$ in equation \eqref{burger's_eq_white_noise_in_time} instead of $|u|^\lambda u_{x^i}$, we consider a stochastic partial differential equation with Lipschitz nonlinear terms; see \eqref{cut_off_equation_large_lambda} and \eqref{nonlinear_cutoff}. Using Lemma \ref{cut_off_lemma_large_lambda}, we find a candidate of a local solution to \eqref{burger's_eq_white_noise_in_time}.

\begin{lemma} 
\label{cut_off_lemma_large_lambda}
Let $\lambda\in(0,\infty)$, $T\in(0,\infty)$, $m\in \bN$, $\kappa\in[0,1)$, and $p \geq 2$. Suppose Assumptions \ref{assumptions_on_coefficients_deterministic_and_linear_part} and \ref{assumptions_on_coefficients_deterministic_and_nonlinear_part} hold. Then, for a bounded stopping time $\tau\leq T$ and nonnegative initial data $u_0\in U_{p}^{1-\kappa}$, there exists a unique solution $u_m \in \cH_p^{1-\kappa}(\tau)$ such that $u_m$ satisfies the equation 
\begin{equation} 
\label{cut_off_equation_large_lambda}
du = \left(a^{ij}u_{x^ix^j} + b^i u_{x^i} + cu + \bar{b}^i\left((u_+)^{1+\lambda} h_m(u)\right)_{x^i}  \right) dt + \sigma^k(u) dw_t^k,\quad (t,x)\in(0,\tau)\times\bR^d
\end{equation}
with $u(0,\cdot) = u_0(\cdot)$ and $u_m\geq0$. 
\end{lemma}
\begin{proof}
For $u,v\in \bR$, we have
\begin{equation}
\label{nonlinear_cutoff}
\left| u_+^{1+\lambda}h_m(u) - v_+^{1+\lambda}h_m(v) \right| \leq N_m|u-v|.
\end{equation}
Indeed,
\begin{equation} \label{lipschitz_check_cut_off}
\begin{aligned}
&\left| u_+^{1+\lambda}h_m(u) - v_+^{1+\lambda}h_m(v) \right| \\
&=
\begin{cases}
\,\,\left| u^{1+\lambda}h_m(u) - v^{1+\lambda}h_m(v) \right| \leq N_m|u-v|  &\text{ if}\quad u,v\geq0,\\
\,\,u^{1+\lambda}h_m(u) \leq (2m)^\lambda u \leq (2m)^\lambda (u-v)\leq N_m|u-v| &\text{ if}\quad u\geq0, v < 0,\\
\,\,v^{1+\lambda}h_m(v) \leq (2m)^\lambda v \leq (2m)^\lambda (v-u)\leq N_m|u-v| &\text{ if}\quad u<0, v\geq0,\\
\,\,0\leq N_m|u-v| &\text{ if}\quad u,v < 0.
\end{cases}
\end{aligned}
\end{equation}
Let $u,v\in \bH_p^{1-\kappa}(\tau)$. Then, by Lemmas \ref{prop_of_bessel_space} \eqref{pointwise_multiplier}, \eqref{bounded_operator}, \eqref{norm_bounded} and \eqref{nonlinear_cutoff}, we have
\begin{equation} 
\label{cut_off_lower_order_cal_1}
\begin{aligned}
&\left\| \bar b^i(t,\cdot) \left((u_+(t,\cdot))^{1+\lambda} h_m(u(t,\cdot))\right)_{x^i} - \bar b^i(t,\cdot)\left((v_+(t,\cdot))^{1+\lambda} h_m(v(t,\cdot))\right)_{x^i} \right\|^p_{H_p^{-1-\kappa}} \\
&\quad \leq N\left\| \left((u_+(t,\cdot))^{1+\lambda} h_m(u(t,\cdot)) - (v_+(t,\cdot))^{1+\lambda} h_m(v(t,\cdot))\right)_{x^i} \right\|^p_{H_p^{-1-\kappa}} \\
&\quad \leq N\left\| (u_+(t,\cdot))^{1+\lambda} h_m(u(t,\cdot)) - (v_+(t,\cdot))^{1+\lambda} h_m(v(t,\cdot)) \right\|^p_{L_p} \\
&\quad\leq N \| u(t,\cdot) - v(t,\cdot) \|_{L_p}^p
\end{aligned}
\end{equation}
and
\begin{equation} \label{cut_off_lower_order_cal_2}
\begin{aligned}
\left\| \sigma(t,\cdot,u(t,\cdot)) - \sigma(t,\cdot,v(t,\cdot)) \right\|_{H^{-\kappa}_p(\ell_2)}^p 
&\leq N\left\| \sigma(t,\cdot,u(t,\cdot)) - \sigma(t,\cdot,v(t,\cdot)) \right\|_{L_p(\ell_2)}^p \\
&\leq N \| u(t,\cdot) - v(t,\cdot) \|^p_{L_p}\\
\end{aligned}
\end{equation}
on $(\omega,t)\in\opar0,\tau\cbrk$. By integrating with respect to $(\omega,t)$ and Lemma \ref{prop_of_bessel_space} \eqref{multi_ineq}, we have
\begin{equation} \label{estimate_for_cut_off_equation}
\begin{aligned}
&\left\| \bar{b}^i\left((u_+)^{1+\lambda} h_m(u)\right)_{x^i} - \bar{b}^i\left((v_+)^{1+\lambda} h_m(v)\right)_{x^i} \right\|^p_{\bH_p^{-1-\kappa}(\tau)} + \left\| \sigma(u) - \sigma(v) \right\|_{\bH^{-\kappa}_p(\tau,\ell_2)}^p \\
&\leq N \| u-v \|_{\bL_p(\tau)}^p \\
&\leq \ep \| u-v \|_{\bH^{1-\kappa}_p(\tau)}^p +  N \| u-v \|_{\bH^{-\kappa}_p(\tau)}^p,
\end{aligned}
\end{equation}
where $N = N(\ep,m,d,p,\kappa,K,T)$.
Therefore, according to Theorem \ref{theorem_nonlinear_case}, there exists a unique solution $u_m\in \cH_p^{1-\kappa}(\tau)$ to equation \eqref{cut_off_equation_large_lambda} using initial data $u_0$.

Next, we show $u_m\geq0$. 
Without loss of generality, we may assume $\kappa = 0$. Indeed, using mollification, we have a sequence of functions $\{u^n_0 \in  U_p^{1}:u^n_0\geq 0,n\in\bN\}$ such that $u^n_0 \to u_0$ in $U^{1-\kappa}_{p}$. By the first assertion and the hypothesis, there exists a unique solution  $u^n_m\in \cH_p^1(\tau)$ to equation \eqref{cut_off_equation_large_lambda} using initial data $u_m^n(0,\cdot) = u_0^n(\cdot)$ and $u_m^n\geq0$. Note that $u^n_m\in\cH_p^{1-\kappa}(\tau)$. Then, for $t \in (0,T)$, Theorem \ref{theorem_nonlinear_case} and \eqref{estimate_for_cut_off_equation} yield
\begin{equation*}
\begin{aligned}
&\|u_m - u_m^n\|_{\cH_p^{1-\kappa}(\tau\wedge t)}^p \\
&\leq N\|u_0 - u_0^n\|_{U_p^{1-\kappa}}^p + N\left\|\bar{b}^i \left((u_{m+})^{1+\lambda} h_m(u_m)\right)_{x^i} - \bar{b}^i \left((u_{m+}^n)^{1+\lambda} h_m(u_{m}^n)\right)_{x^i} \right\|_{\bH_p^{-1-\kappa}(\tau\wedge t)}^p \\
& \quad + N\left\| \sigma(u_{m}) - \sigma(u_{m}^n) \right\|_{\bH^{-\kappa}_p(\tau\wedge t,\ell_2)}^p \\
&\leq N\|u_0 - u_0^n\|_{U_p^{1-\kappa}}^p + \frac{1}{2} \| u-v \|_{\bH^{1-\kappa}_p(\tau)}^p +  N \| u-v \|_{\bH^{-\kappa}_p(\tau)}^p,
\end{aligned}
\end{equation*}
where $N = N(m,d,p,\kappa,K,T)$. According to Theorem \ref{embedding} \eqref{gronwall_type_ineq}, we have
\begin{equation*}
\|u_m - u_m^n\|_{\cH_p^{1-\kappa}(\tau)}^p \leq N \|u_0 - u_0^n\|_{U_p^{1-\kappa}}^p,
\end{equation*}
where $N = N(m,d,p,\kappa,K,T)$. Since $N$ is independent of $n$, by letting $n\to\infty$, $u_m^n\to u_m$ in $\cH_p^{1-\kappa}(\tau)$. Since $u_m^n \geq0$ for almost every $(\omega,t,x)$, $u_m \geq0$ for almost every $(\omega,t,x)$.

Now, we show that if $u_m\in\cH_p^1(\tau)$ is a unique solution to equation \eqref{cut_off_equation_large_lambda} using the initial data $u_m(0,\cdot) = u_0\geq0$, then $u_m\geq0$. By applying \cite[Theorem 1.1]{krylov2007maximum} with $f = \bar b^i\left((u_{m+})^{1+\lambda}h_m(u_m)\right)_{x^i}$, we obtain $u_m \geq 0$ for all $t\leq\tau$  almost surely.
The lemma is proved. 
\end{proof}

\begin{lemma} \label{local_existence}

Let $\lambda\in(0,\infty)$ and $T\in(0,\infty)$. Suppose $\kappa\in[0,1)$ and $p > 2$ satisfy
\begin{equation*}
p > \frac{d+2}{1-\kappa}.
\end{equation*}
Then, for a bounded nonzero stopping time $\tau \leq T$ and nonnegative initial data $u_0\in U_{p}^{1-\kappa}$, there exists a nonzero stopping time $\tau'\leq \tau$ such that equation \eqref{burger's_eq_white_noise_in_time} has a unique nonnegative solution $u\in \cH_{p}^{1-\kappa}(\tau')$.

\end{lemma}
\begin{proof}
Since $p > \frac{d+2}{1-\kappa} \geq d+2$, according to Lemma \ref{prop_of_bessel_space} \eqref{sobolev-embedding}, we have 
\begin{equation*} 
\bE\sup_{x\in \bR^d}|u_0(x)|^p \leq N\bE|u_0|_{C^{1-\frac{d+2}{p}}(\bR^d)}^p \leq N\bE\| u_0 \|_{H_p^{1-2/p}}^p = N\| u_0 \|_{U_p^1}^p < \infty.
\end{equation*} 
This relationship implies
\begin{equation*}
P\left( \cup_{n\in\bN} \left\{ \omega\in\Omega: \sup_{x\in\bR^d}| u_0(x) | < n \right\} \right) = 1.
\end{equation*}
Since $P(\left\{ \omega\in\Omega: \tau > 0 \right\}) > 0$, there exists $m\in\{2,3,\dots\}$ such that 
\begin{equation}
\label{nonzero_condi}
P\left(\left\{ \omega\in\Omega: \tau > 0 \right\}\cap\left\{ \omega\in\Omega: \sup_{x\in\bR^d}| u_0(x) | < m-1 \right\} \right) > 0.
\end{equation}
For this $m$, by Lemma \ref{cut_off_lemma_large_lambda}, there exists $u_{m}\in\cH_p^{1-\kappa}(\tau)$ such that $u_{m}$ satisfies \eqref{cut_off_equation_large_lambda} and $u_{m}\geq0$. Since $p > \frac{d+2}{1-\kappa}$, according to Corollary \ref{embedding_corollary}, we have $u_m\in C([0,\tau];C(\bR^d))$ (a.s.). Thus, if we define
\begin{equation*}
\tau_{m}:=\inf\left\{ t \geq 0 : \sup_{x\in\bR^d} |u_{m}(t,x)| \geq m \right\}\wedge \tau,
\end{equation*}
then $\tau_{m}$ is a well-defined stopping time and the solution $u_{m}$ satisfies equation \eqref{burger's_eq_white_noise_in_time} for all $t \leq \tau_{m}$ almost surely. Additionally, note that the stopping time $\tau_{m}$ is nonzero. Indeed, since $\left\{ \omega\in\Omega: \tau>0,\,\sup_{x}| u_0(x) | < m - 1 \right\}\subset\{ \omega\in\Omega : \tau_{m} > 0 \}$, \eqref{nonzero_condi} implies
\begin{equation*}
0 < P\left( \left\{ \omega:\tau>0,\, \sup_{x\in\bR^d}| u_0(x) | < m - 1 \right\} \right) \leq P(\left\{\omega\in\Omega:\tau_{m}>0\right\}).
\end{equation*}
 By taking $\tau' = \tau_{m}$, the lemma is proved.
\end{proof}

Next, the lemma provides $\bL_q(\tau)$ ($q>p$) bound of the solution. To prove the existence of a global solution, we insert in $q = p(1+\lambda)$. The $\bL_{p(1+\lambda)}(\tau)$ bound of the solution is used to extend the local existence time; see \eqref{nonlinear term estimate}.

\begin{lemma} \label{Lq_bound_infinite_noise}
Let $\tau\leq T$ be a bounded stopping time, $p > d + 2$, and $q > p$. Suppose $u\in\cH^{1}_{p,loc}(\tau)$ is a nonnegative solution to equation \eqref{burger's_eq_white_noise_in_time} with $u_0\in U_p^1\cap L_q(\Omega\times\bR^d)$.  Then, we have $u\in \bL_q(\tau)$  and
\begin{equation} 
 \|u\|_{\bL_q(\tau)}^{q} = \bE\int_0^\tau\int_{\bR^d}|u(t,x)|^q dxds \leq N\bE\int_{\bR^d} |u_0(x)|^q dx,
\end{equation}
where $N = N(d,q,K,T)$.
\end{lemma}
\begin{proof}
According to Definition \ref{definition_of_sol_space} \eqref{def_of_local_sol_space}, there exists a sequence of bounded stopping times $\{ \tau_n : n\in\bN \}$ such that $\tau_n\uparrow\tau$ (a.s.) as $n\to\infty$ and $u\in\cH_p^1(\tau_n)$. By Corollary \ref{embedding_corollary}, we have $u\in C([0,\tau];C(\bR^d))$ (a.s.) and
\begin{equation*}
\bE\sup_{t \leq \tau_n}\sup_{x\in\bR^d}|u(t,x)|^p  < \infty,
\end{equation*}
where $N = N(d,p,T)$. Thus, if we fix $n\in\bN$ and set
$$ \tau_{m,n} := \inf\left\{ t\geq0 : \sup_{x\in\bR^d}|u(t,x)| \geq m \right\}\wedge \tau_n,\quad (m\in\bN)
$$ 
then $\tau_{m,n}$ is a well-defined stopping time and $\tau_{m,n}\uparrow\tau_n$ (a.s.) as $m\to\infty$. Observe that $u\in\bL_{p_0}(\tau_{m,n})$ for any $p_0 \geq p$. Indeed,
\begin{equation} \label{for_any_p0}
\bE\int_0^{\tau_{m,n}}\int_{\bR^d} |u(t,x)|^{p_0} dxdt \leq m^{p_0-p}\bE\int_0^{\tau_{m,n}}\int_{\bR^d} |u(t,x)|^p dxdt = m^{p_0-p}\| u \|_{\bL_p(\tau_{m,n})}^p < \infty.
\end{equation}
We have to find an appropriate test function to employ It\^o's formula with equation \eqref{burger's_eq_white_noise_in_time}. A nonnegative function $\zeta\in C_c^\infty(\bR^d)$ is chosen such that $\int_{\bR^d}\zeta(x)dx = 1$, and $\zeta_{\ep}(x) := \ep^{-d}\zeta(\ep^{-1}x)$ is set for $\ep>0$. Define $f^{(\ep)}(x) := \int_{\bR^d}f(y)\zeta_\ep(x-y)dy$ for $f\in L_{1,loc}(\bR^d)$ and $\ep>0$. Let $M>0$, which will be specified later. 
For $\ep>0$ and $t > 0$, consider equation \eqref{burger's_eq_white_noise_in_time} with the test function $\zeta_\ep(x-\cdot)$. Then, we have
\begin{equation}
\label{eq:employing mollification}
\begin{aligned}
&u^{(\ep)}(t,x) - u_0^{(\ep)}(x) \\
&\quad= \int_0^t \int_{\bR^d} u(s,y) \left( a^{ij}(s,y)\zeta_\ep(x-y) \right)_{y^iy^j} - u(s,y)\left(b^i(s,y)\zeta_\ep(x-y)\right)_{y^i} dxds \\
&\quad\quad + \int_0^t \int_{\bR^d} c(s,y)u(s,y)\zeta_{\ep}(x-y) dyds + \int_0^t \int_{\bR^d} \sigma^k(s,y,u(s,y))\zeta_\ep(s,x-y) dx dw_s^k \\
&\quad= \int_0^t \left(a_{x^ix^j}^{ij}u\right)^{(\ep)}(s,x) -2\left(a_{x^j}^{ij}u\right)^{(\ep)}_{x^i}(s,x) + \left( a^{ij}u \right)^{(\ep)}_{x^ix^j}(s,x) - \left( b^i_{x^i}u \right)^{(\ep)}(s,x) ds \\
&\quad\quad + \int_0^t \left( b^iu \right)^{(\ep)}_{x^i}(s,x) + (cu)^{(\ep)}(s,x)ds + \int_0^t \left( \sigma^k(s,\cdot,u(s,\cdot)) \right)^{(\ep)}(x) dw_s^k
\end{aligned}
\end{equation}
Thus, by applying It\^o's formula  integrations, and integration by parts on \eqref{eq:employing mollification}, we have
\begin{equation*} 
\begin{aligned}
&\bE e^{-M\tau_{m,n}\wedge t}\int_{\bR^d} \left|u^{(\ep)}(\tau_{m,n}\wedge t,x)\right|^q dx  - \bE \int_{\bR^d} \left|u_{0}^{(\ep)}(x)\right|^q dx \\
& = -q(q-1)\bE \int_0^{\tau_{m,n}\wedge t} \int_{\bR^d} \left|u^{(\ep)}(s,x)\right|^{q-2}u^{(\ep)}_{x^i}(s,x) \bigg(  \left(a^{ij}u\right)^{(\ep)}_{x^j}(s,x)   \\
& \phantom{ = -q(q-1)\bE \int_0^{\tau_{m,n}\wedge t} \int_{\bR^d} \left|u^{(\ep)}(s,x)\right|^{q-2}} + 2\left(a^{ij}_{x^j}u\right)^{(\ep)}(s,x) - \left(b^{i}u\right)^{(\ep)}(s,x) \bigg) e^{-Ms}dxds \\
&\quad +q\bE \int_0^{\tau_{m,n}\wedge t} \int_{\bR^d} \left|u^{(\ep)}(s,x)\right|^{q-1}\bigg[ \left(a^{ij}_{x^ix^j}u\right)^{(\ep)}(s,x) - \left(b^{i}_{x^i}u\right)^{(\ep)}(s,x) \\
&\phantom{\quad +q\bE \int_0^{\tau_{m,n}\wedge t} \int_{\bR^d} \left|u^{(\ep)}(s,x)\right|^{q-1}\bigg[ \left(a^{ij}_{x^ix^j}u\right)^{(\ep)}(s,x)}+ (cu)^{(\ep)}(s,x)\bigg] e^{-Ms}dxds \\
&\quad -\frac{q(q-1)}{1+\lambda}\bE \int_0^{\tau_{m,n}\wedge t} \int_{\bR^d} \left|u^{(\ep)}(s,x)\right|^{q-2}u_{x^i}^{(\ep)}(s,x) \left( \bar{b}^i\left(s,\cdot\right)(u(s,\cdot))^{1+\lambda}\right)^{(\ep)}(x)  e^{-Ms}dxds \\
&\quad +\frac{q(q-1)}{2}\bE \int_0^{\tau_{m,n}\wedge t} \int_{\bR^d} \left|u^{(\ep)}(s,x)\right|^{q-2} \sum_k \left( \left(\sigma^k(s,\cdot,u(s,\cdot))\right)^{(\ep)}(x) \right)^2 e^{-Ms}dxds \\
&\quad -M\bE \int_0^{\tau_{m,n}\wedge t}\int_{\bR^d} \left|u^{(\ep)}(s,x)\right|^q e^{-Ms}dxds.
\end{aligned}
\end{equation*}
By letting $\ep\downarrow0$, \eqref{boundedness_of_deterministic_coefficients_white_noise_in_time} and \eqref{boundedness_of_stochastic_coefficients_white_noise_in_time_large_lambda} imply
\begin{equation} \label{applying_ito_formula_for_Lq_bound_infinite_noise_1}
\begin{aligned}
&\bE e^{-M\tau_{m,n}\wedge t}\int_{\bR^d} \left|u(\tau_{m,n}\wedge t,x)\right|^q dx  - \bE \int_{\bR^d} \left|u_{0}(x)\right|^q dx \\
&\quad \leq -q(q-1)\bE \int_0^{\tau_{m,n}\wedge t} \int_{\bR^d} \left|u(s,x)\right|^{q-2}a^{ij}(s,x)u_{x^i}(s,x)u_{x^j}(s,x) e^{-Ms}dxds \\
&\quad\quad +N\sum_{i}\bE \int_0^{\tau_{m,n}\wedge t} \int_{\bR^d} \left|u(s,x)\right|^{q-1}\left| u_{x^i}(s,x) \right|  e^{-Ms}dxds \\
\end{aligned}
\end{equation}
\begin{equation*}
\begin{aligned}
&\quad\quad -\frac{q(q-1)}{1+\lambda}\bE \int_0^{\tau_{m,n}\wedge t} \int_{\bR^d}\bar{b}^i\left(s,\bar x^i\right) \left|u(s,x)\right|^{q+\lambda-1}u_{x^i}(s,x)    e^{-Ms}dxds \\
&\quad\quad +(N-M)\bE \int_0^{\tau_{m,n}\wedge t}\int_{\bR^d} \left|u(s,x)\right|^q e^{-Ms}dxds,
\end{aligned}
\end{equation*}
where $N = N(d,q,K)$. Note that 
\begin{equation} \label{Lq_bound_nonlinear_noise_zero}
\bE \int_0^{\tau_{m,n}\wedge t} \int_{\bR^d} \bar{b}^i\left(s,\bar x^i\right)(u(s,x))^{q+\lambda-1} u_{x^i}(s,x) e^{-Ms}dxds = 0.
\end{equation}
Indeed,
\begin{equation*}
\begin{aligned}
&\bE \int_0^{\tau_{m,n}\wedge t} \int_{\bR^d} \bar{b}^i\left(s,\bar x^i\right)(u(s,x))^{q+\lambda-1} u_{x^i}(s,x)   e^{-Ms}dxds \\
& = \bE \int_0^{\tau_{m,n}\wedge t}  \int_{\bR^{d-1}}\bar{b}^i\left(s,\bar x^i\right) \int_{\bR}(u(s,x))^{q+\lambda-1} u_{x^i}(s,x) - (u^{(\ep)}(s,x))^{q+\lambda-1} u^{(\ep)}_{x^i}(s,x) dx^i d\bar x^i ds
\\
& \leq N\bE \int_0^{\tau_{m,n}\wedge t}\sum_{i=1}^d\left| \int_{\bR^d} (u(s,x))^{q+\lambda-1}u_{x^i}(s,x) - \left(u^{(\ep)}(s,x)\right)^{q+\lambda-1} u^{(\ep)}_{x^i}(s,x) dx \right|ds \\
&\leq N\left(\| u \|^{q+\lambda-2}_{\bL_{\frac{p(q+\lambda-2)}{p-2}}(\tau_{m,n})} + \left\| u^{(\ep)} \right\|^{q+\lambda-2}_{\bL_{\frac{p(q+\lambda-2)}{p-2}}(\tau_{m,n})}\right)\left\| u - u^{(\ep)} \right\|_{\bL_p(\tau_{m,n})}\| u_{x} \|_{\bL_p(\tau_{m,n})} \\
&\quad + N\left\| u^{(\ep)} \right\|^{q+\lambda-1}_{\bL_{\frac{p(q+\lambda-1)}{p-1}}(\tau_{m,n})}\left\| u_x - u_x^{(\ep)} \right\|_{\bL_p(\tau_{m,n})},
\end{aligned}
\end{equation*}
where $N = N(\lambda,d,p,q,K)$.
Note that $q+\lambda-2 > p-2$, $q+\lambda-1 > p-1$. Since $u\in \bL_{p_0}(\tau_{m,n})$ for any $p_0 \geq p$, $u\in \bH_p^1(\tau_{m,n})$, and $\ep > 0$ is arbitrary, \eqref{Lq_bound_nonlinear_noise_zero} holds. 

Observe that, by \eqref{ellipticity_of_leading_coefficients_white_noise_in_time}, we have
\begin{equation} \label{applying_ellipiticity_for_Lq_bound_infinite_noise_1}
\begin{aligned}
\int_{\bR^d}|u(s,x)|^{q-2}a^{ij}(s,x)u_{x^i}(s,x)u_{x^j}(s,x)dx \geq K^{-1} \int_{\bR^d}|u(s,x)|^{q-2}|u_x(s,x)|^2 dx
\end{aligned}
\end{equation}
for all $s\leq \tau_{m,n}\wedge t$ almost surely. In addition, by Young's inequality, there exists $N = N(d,q,K)$ such that
\begin{equation} \label{applying_Young's_ineq_for_Lq_bound_infinite_noise_1}
\begin{aligned}
&N\sum_i\bE \int_0^{\tau_{m,n}\wedge t} \int_{\bR^d} \left|u(s,x)\right|^{q-1}\left| u_{x^i}(s,x) \right|  e^{-Ms}dxds \\
&\leq \frac{1}{2}q(q-1)\bE \int_0^{\tau_{m,n}\wedge t} \int_{\bR^d} \left|u(s,x)\right|^{q-2}\left| u_{x}(s,x) \right|^2  e^{-Ms}dxds \\
& \quad\quad + N\bE \int_0^{\tau_{m,n}\wedge t} \int_{\bR^d} \left|u(s,x)\right|^{q}  e^{-Ms}dxds.
\end{aligned}
\end{equation}
By applying \eqref{Lq_bound_nonlinear_noise_zero}, \eqref{applying_ellipiticity_for_Lq_bound_infinite_noise_1}, and \eqref{applying_Young's_ineq_for_Lq_bound_infinite_noise_1} to \eqref{applying_ito_formula_for_Lq_bound_infinite_noise_1}, we have
\begin{equation*} 
\begin{aligned}
&\bE e^{-M\tau_{m,n}\wedge t}\int_{\bR^d} \left|u(\tau_{m,n}\wedge t,x)\right|^q dx  - \bE \int_{\bR^d} \left|u_{0}(x)\right|^q dx \\
&\quad \leq (N-M)\bE \int_0^{\tau_{m,n}\wedge t}\int_{\bR^d} \left|u(s,x)\right|^q e^{-Ms}dxds,
\end{aligned}
\end{equation*}
where $N = N(d,q,K)$.  Thus, by letting $M = 2N$, we derive
\begin{equation}
\label{before final estimate}
\begin{aligned}
\bE e^{-M\tau_{m,n}\wedge t}\int_{\bR^d} |u(\tau_{m,n} \wedge t,x)|^q dx \leq \bE \int_{\bR^d} |u_0(x)|^q dx.
\end{aligned}
\end{equation}
Since $N$ is independent of $m,n$, by letting $m\to\infty$, $n\to\infty$ in order, we have
\begin{equation*}
\bE e^{-M\tau\wedge t}\int_{\bR^d} |u(\tau\wedge t,x)|^q dx \leq \bE \int_{\bR^d}|u_0(x)|^q dx.
\end{equation*}
Note that
\begin{equation*}
\begin{aligned}
\bE \int_0^{\tau} \int_{\bR^d} |u(t,x)|^q dxdt &\leq N\bE \int_0^{\tau} e^{-Mt}\int_{\bR^d} |u(t,x)|^q dxdt \\
&\leq N\int_0^{T}\bE  e^{-M\tau\wedge t}\int_{\bR^d} |u(\tau\wedge t,x)|^q dxdt \\
&\leq N\bE \int_{\bR^d} |u_0(x)|^q dx,
\end{aligned}
\end{equation*}
where $N = N(d,q,K,T)$. The lemma is proved.
\end{proof}

\begin{remark}
It should be noted that the constant $N$ introduced in Lemma \ref{Lq_bound_infinite_noise} depends on $T$, not $\tau$.
\end{remark}

Now, we prove Theorem \ref{theorem_burgers_eq_white_noise_in_time_large_lambda}.

\begin{proof}[\bf{Proof of Theorem \ref{theorem_burgers_eq_white_noise_in_time_large_lambda}}]

{\it Step 1. (Uniqueness). }
Suppose $u,\bar u\in \cH_{p,loc}^{1}$ are nonnegative solutions to equation \eqref{burger's_eq_white_noise_in_time}. By Definition \ref{definition_of_sol_space} \eqref{def_of_local_sol_space}, there are bounded stopping times $\tau_n$, $n = 1,2,\cdots$ 
such that
\begin{equation*}
\tau_n\uparrow\infty\quad\mbox{and}\quad u, \bar u \in \cH_{p}^{1}(\tau_n).
\end{equation*}
Let $n\in\bN$ be fixed. Since $p > d+2$, by Corollary \ref{embedding_corollary}, we have $u,\bar{u} \in C([0,\tau_n];C(\bR^d))$ (a.s.) and
\begin{equation}
\label{embedding in the proof of uniqueness_infinite_noise_1}
\bE\sup_{t\leq\tau_n}\sup_{x\in\bR^d}|u(t,x)|^p + \bE\sup_{t\leq\tau_n}\sup_{x\in\bR^d}|\bar u(t,x)|^p < \infty.
\end{equation}
For $m \in \bN$, define
\begin{equation*}
\begin{gathered}
\tau_{m,n}^1:=\inf\left\{t\geq0:\sup_{x\in\bR^d}|u(t,x)|> m\right\}\wedge\tau_n, \quad \tau_{m,n}^2:=\inf\left\{t\geq0:\sup_{x\in\bR^d}|\bar u(t,x)|> m\right\}\wedge\tau_n,
\end{gathered}
\end{equation*}
and 
\begin{equation} 
\label{stopping_time_cutting}
\tau_{m,n}:=\tau_{m,n}^1\wedge\tau_{m,n}^2.
\end{equation}
Due to \eqref{embedding in the proof of uniqueness_infinite_noise_1}, $\tau_{m,n}^1$ and $\tau_{m,n}^2$ are well-defined stopping times, and thus, $\tau_{m,n}$ is a stopping time.
Observe that $u,\bar u\in\cH_p^{1}(\tau_{m,n})$ and  $\tau_{m,n}\uparrow \tau_n$ as $m\to\infty$ almost surely. Fix $m\in\bN$. $u,\bar u\in\cH_p^{1}(\tau_{m,n})$ are solutions to equation
\begin{equation*}
dv = \left(  a^{ij}v_{x^ix^j} + b^i v_{x^i} + cv + \bar{b}^i \left( v_+^{1+\lambda}h_m(v) \right)_{x^i}  \right)\,dt + \sigma^k(v)  dw^k_t, \, 0<t\leq\tau_{m,n};\, v(0,\cdot)=u_0.
\end{equation*}
According to the uniqueness result in Lemma \ref{cut_off_lemma_large_lambda}, we conclude that $u=\bar u$ in $\cH_p^{1}(\tau_{m,n})$ for each $m\in\bN$. The monotone convergence theorem yields $u=\bar u$ in $\cH_p^{1}(\tau_n)$, which implies $u = \bar u$ in $\cH_{p,loc}^1$.
\\

{\it Step 2. (Existence). }
The motivation for the proof follows from that in \cite[Theorem 2.9]{kim2018regularity}.  Let $\tau\leq T$ be a nonzero bounded stopping time. First, we show that there exists a solution $u\in\cH_p^1(\tau)$ to equation \eqref{burger's_eq_white_noise_in_time}. We define that
\begin{equation*}
\begin{aligned}
\Pi := \bigg\{ \text{stopping times } \tau_\alpha\leq\tau:  \text{ equation }\eqref{burger's_eq_white_noise_in_time} &\text{ has a nonnegative solution }\\
&\quad \quad \quad u\in\cH_p^{1}(\tau_\alpha)\text{ with }u(0) = u_0 \bigg\}.
\end{aligned}
\end{equation*}
With the help of Lemma \ref{local_existence}, $\Pi$ is nonempty. Observe that if $\tau_{\alpha_1},\tau_{\alpha_2}\in\Pi$, then $\tau_{\alpha_1}\vee\tau_{\alpha_2}\in\Pi$. Indeed, if $u^1\in\cH_{p}^1(\tau_{\alpha_1})$, $u^2\in\cH_{p}^1(\tau_{\alpha_2})$ are solutions; then, the uniqueness in $\cH_p^1(\tau_{\alpha_1}\wedge\tau_{\alpha_2})$ implies $u^1 = u^2$ on $(0,\tau_{\alpha_1}\wedge\tau_{\alpha_2})$. On the other hand, we define $u = u^1$ if $\tau_{\alpha_1} \geq \tau_{\alpha_2}$ and $u = u_2$ otherwise. Then, $u\in\cH_p^1(\tau_{\alpha_1}\vee\tau_{\alpha_2})$, which implies $\tau_{\alpha_1}\vee\tau_{\alpha_2}\in\Pi$.

Define $r:=\sup_{\tau_\alpha\in\Pi}\bE\tau_{\alpha}$. Then, there exists a set of bounded stopping times $\{\tau_n:n\in\bN\}$ such that $\bE\tau_n \to r$. Without loss of generality, we may assume that $\tau_n$ is nondecreasing. If we set $\bar{\tau}:= \limsup_{n\to\infty}\tau_n$, then $\bar{\tau}$ is a well-defined stopping time since the filtration $\cF_t$ is right continuous. Notice that $\bE\bar{\tau} = r$ by monotone convergence theorem.

We show that there exists a solution $u$ in $\cH_p^{1}(\bar{\tau})$. Since $\tau_n$ is nondecreasing, due to the uniqueness, there exists $u\in\cH_{p,loc}^{1}(\bar{\tau})$ such that $u(0,\cdot) = u_0$. Note that for each $n\in\bN$, $u\in\cH_p^1(\tau_n)$ satisfies the equation
\begin{equation} \label{local_sol_satisfies_eq}
du = \left(a^{ij}u_{x^ix^j} + b^{i}u_{x^i} + cu + \bar{b}^i u^\lambda u_{x^i}\right)dt + \sigma^k (u) dw_t^k,\quad t\leq \tau_n\,;\quad u(0,\cdot) = u_0(\cdot).
\end{equation}
By Lemma \ref{prop_of_bessel_space} \eqref{bounded_operator}, \eqref{pointwise_multiplier} and  Lemma \ref{Lq_bound_infinite_noise}, we have
\begin{equation} \label{nonlinear term estimate}
\begin{aligned}
\left\| \bar{b}^i u^\lambda u_{x^i} \right\|^p_{\bH_p^{-1}(\tau_n)} & \leq N\bE\int_0^{\tau_n} \int_{\bR^d} |u(s,x)|^{p(1+\lambda)} dxds \leq N\bE\int_{\bR^d}|u_0(x)|^{p(1+\lambda)} dx,
\end{aligned}
\end{equation}
where $N = N(\lambda,d,p,K,T)$. It should be noted that $N$ is independent of $n$. Therefore, by Theorem \ref{theorem_nonlinear_case} and \eqref{nonlinear term estimate}, we have
\begin{equation*}
\begin{aligned}
\| u \|^p_{\bH_p^1(\tau_n)}
&\leq \| u \|^p_{\cH_p^{1}(\tau_n)} \\
&\leq N\| u_0 \|^p_{U_p^1} + N\left\| \bar{b}^i u^\lambda u_{x^i}  \right\|_{\bH_p^{-1}(\tau_n)}^p \\
&\leq N\| u_0 \|^p_{U_p^1} + N\| u_0 \|_{L_{p(1+\lambda)}(\Omega;L_{p(1+\lambda)}(\bR^d))}^{p(1+\lambda)},
\end{aligned}
\end{equation*}
where $N = N(\lambda,d,p,K,T)$. Since $N$ is independent of $n$, by taking $n\to\infty$, we obtain
\begin{equation}
\label{hp-1-valued_conti_ft}
\| u \|_{\bH^{1}_p(\bar{\tau})} + \left\| \bar{b}^i u^\lambda u_{x^i} \right\|_{\bH_p^{-1}(\bar\tau)} < \infty.
\end{equation}
\eqref{hp-1-valued_conti_ft} implies that the right-hand side of equation \eqref{local_sol_satisfies_eq} is an $H_{p}^{-1}$-valued continuous function on $[0,\bar{\tau}]$. Therefore, $u$ is continuously extendible to $[0,\bar\tau]$, and thus, $u$ satisfies equation \eqref{local_sol_satisfies_eq}  for all $t\leq \bar{\tau}$ almost surely. Therefore, we have $u \in \cH_{p}^{1}(\bar{\tau})$. 

Now, we show $\bar{\tau} = \tau$ (a.s.). Suppose that it is not true. Then, $P\left(\{\omega \in \Omega: \bar{\tau}<\tau \}\right)>0$. We choose $\kappa\in(0,1)$ such that
\begin{equation*}
p > \frac{d+2}{1-\kappa} > d+2.
\end{equation*}
By Lemma \ref{prop_of_bessel_space} \eqref{sobolev-embedding} and Theorem \ref{embedding} \eqref{large-p-embedding}, we have 
\begin{equation*}
\| u(\bar{\tau},\cdot) \|^p_{U_p^{1-\kappa}} = \bE\| u(\bar{\tau},\cdot) \|^p_{H_p^{1-\kappa-2/p}} \leq \bE\sup_{t\leq \bar{\tau}}\| u(t,\cdot) \|^p_{H_p^{1-\kappa-2/p}} \leq N\| u \|_{\cH_p^1(\bar{\tau})}^p < \infty.
\end{equation*}
Set $\bar{u}_0(\cdot) = u(\bar{\tau},\cdot)$, $\cF_t^{\bar{\tau}} := \cF_{t+\bar{\tau}}$, and $\bar{w}_t^k:=w_{t+\bar{\tau}}^k - w_{\bar{\tau}}^k$. Then, $\bar{w}_t^k$ are independent Wiener processes relative to $\cF_t^{\bar{\tau}}$, $\bar{u}_0$ is $\cF_0^{\bar{\tau}}$-measurable, and $\tilde{\tau}:=\tau - \bar{\tau}$ is a nonzero bounded stopping time with respect to $\cF_t^{\bar{\tau}}$. Consider equation
\begin{equation} \label{extended_sol}
d\bar{v} = \left(a^{ij}\bar{v}_{x^ix^j} + b^i \bar{v}_{x^i} + c\bar{v} + \bar{b}^i\bar{v}^\lambda \bar{v}_{x^i}\right) dt + \sigma^k(\bar{v}) d\bar w_t^k,\quad t\leq \tilde{\tau}\,;\quad \bar{v}(0,\cdot) = \bar{u}_0.
\end{equation}
Then, by Lemma \ref{local_existence}, there exists a nonzero stopping time $\tilde{\tau}' \leq \tilde\tau$ (with respect to $\cF_t^{\bar{\tau}}$) such that equation \eqref{extended_sol} has an $\cF_{t}^{\bar{\tau}}$-adapted nonnegative solution $v\in\cH_p^{1-\kappa}(\tilde{\tau}')$.
Set $\tau_0 := \bar{\tau} + \tilde{\tau}'$. Note that $\tau_0 \leq \tau$ is a bounded stopping time and $\bE\tau_0 > r$. Define
\begin{equation*}
w(t,x) := 
\begin{cases}
u(t,x),\quad&\text{if}\quad t\leq \bar{\tau}, \\ 
v(t - \bar{\tau},x),\quad&\text{if}\quad  \bar{\tau} < t \leq \tau_0.
\end{cases}
\end{equation*}
Then, $w$ satisfies equation \eqref{burger's_eq_white_noise_in_time} for $t\leq \tau_0$ almost surely and $w\in \cH_p^{1-\kappa}(\tau_0)$. For $m\in\bN$, define 
\begin{equation*}
\tau_m := \inf\left\{ t\geq0 : \sup_{x\in\bR^d}|w(t,x)| \geq m \right\} \wedge \tau_0.
\end{equation*}
Due to Corollary \ref{embedding_corollary}, $\tau_m$ is a well-defined stopping time and $\tau_m\uparrow \tau_0$ (a.s.) as $m\to\infty$. Note that $\tau_m\leq T$ for any $m\in\bN$. Fix $m\in\bN$. Observe that $w\in\cH_{p}^{1-\kappa}(\tau_m)$ is nonnegative and satisfies \eqref{cut_off_equation_large_lambda} for all $t\leq \tau_m$ almost surely. On the other hand, according to Lemma \ref{cut_off_lemma_large_lambda}, there exists $\bar{w}_m\in\cH_p^{1}(\tau_m)$ satisfying the equation
\begin{equation} \label{equation_regularity_extending_infinite_noise_1}
d\bar{w}_m = \left(a^{ij}\bar{w}_{mx^ix^j} + b^i \bar{w}_{mx^i} + c\bar{w}_m +  \bar{b}^i\left((\bar{w}_{m+})^{1+\lambda} h_m(\bar{w}_m)\right)_{x^i}  \right) dt + \sigma^k(\bar{w}_m) dw_t^k
\end{equation}
with $\bar{w}_m(0,\cdot) = u_0(\cdot)$ for all $t\leq \tau_m$ almost surely. Since $w\in\cH_p^{1-\kappa}(\tau_m)$ satisfies equation \eqref{equation_regularity_extending_infinite_noise_1}, according to the uniqueness of Lemma \ref{cut_off_lemma_large_lambda}, we have $w = \bar w_m$ in $\cH_p^{1-\kappa}(\tau_m)$. Thus, $w$ is in $\cH_{p}^1(\tau_m)$ and $w$ satisfies \eqref{burger's_eq_white_noise_in_time} for all $t \leq \tau_m$ almost surely with initial data $w(0,\cdot) = u_0(\cdot)$.  As before, by Theorem \ref{theorem_nonlinear_case}, Lemma \ref{prop_of_bessel_space} \eqref{bounded_operator}, \eqref{pointwise_multiplier}, and  Lemma \ref{Lq_bound_infinite_noise}, we have
\begin{equation*}
\begin{aligned}
\| w \|^p_{\bH_p^1(\tau_m)} 
&\leq N\| w \|^p_{\cH_p^1(\tau_m)} \\
&\leq N\| u_0 \|^p_{U_p^1} + \left\| \bar{b}^i w^\lambda w_{x^i} \right\|_{\bH_p^{-1}(\tau_m)}^p \\
&\leq N\| u_0 \|^p_{U_p^1} +  N\| u_0 \|^{p(1+\lambda)}_{L_{p(1+\lambda)}(\Omega;L_{p(1+\lambda)}(\bR^d))},
\end{aligned}
\end{equation*}
where $N = N(\lambda,d,p,K,T)$. Since $N$ is independent of $m$, by taking $m\to\infty$, we have
\begin{equation*}
\| w \|_{\bH^{1}_p(\tau_0)} + \left\| \bar{b}^i w^\lambda w_{x^i} \right\|_{\bH_p^{-1}(\tau_0)} < \infty.
\end{equation*}
Again, this inequality implies that the right hand side of equation \eqref{burger's_eq_white_noise_in_time} is an $H_{p}^{-1}$-valued continuous function on $[0,\tau_0]$, and thus, $w$ satisfies equation \eqref{burger's_eq_white_noise_in_time} for all $t\leq\tau_0$ almost surely. Therefore, we have $w \in \cH_{p}^{1}(\tau_0)$. 
Since $\bE\tau_0 > r$ and $\tau_0\in \Pi$, this statement implies a contradiction. Thus, $\bar{\tau} = \tau$ (a.s.). 

For each bounded stopping time $\tau$, there exists a solution $u_\tau\in\cH_{p}^1(\tau)$ to equation \eqref{burger's_eq_white_noise_in_time}. Define
\begin{equation*}
u = u_T \quad\text{on}\quad [0,T],
\end{equation*}
where $T\in\bN$. Due to the uniqueness in $\cH_{p}^1(T)$, the function $u$ is well-defined. By Definition \ref{definition_of_sol_space} \eqref{def_of_local_sol_space}, $u\in\cH_{p,loc}^1$. \\

{\it Step 3. (H\"older regularity). } Let $T<\infty$. Since $u\in \cH_p^1(T)$, by employing Corollary \ref{embedding_corollary}, we derive \eqref{holder_regularity_main_theorem}. The theorem is proved.

\end{proof}

\begin{proof}[\bf{Proof of Theorem \ref{uniqueness_in_p_1}}]
Let $q>p$. By Theorem \ref{theorem_burgers_eq_white_noise_in_time_large_lambda}, there exists a unique solution $\bar{u}\in \cH_{q,loc}^{1}$ satisfying equation \eqref{burger's_eq_white_noise_in_time}. In addition, according to \ref{definition_of_sol_space} \eqref{def_of_local_sol_space}, there exists $\tau_n$ such that $\tau_n\to\infty$ (a.s.) as $n\to\infty$, $u\in\cH_p^1(\tau_n)$ and $\bar{u} \in\cH_q^1(\tau_n)$. Fix $n\in\bN$. Since $d+2 < p < q$, we can define $\tau_{m,n}$ $(m\in\bN)$ as in \eqref{stopping_time_cutting}. As in \eqref{for_any_p0}, for any $p_0 > p$, we have
\begin{equation*}
u\in \bL_{p_0}(\tau_{m,n}).
\end{equation*}
Observe that $\bar{b}^i \left( u^{1+\lambda} \right)_{x^i} \in \bH_q^{-1}(\tau_{m,n})$. Indeed, by \eqref{boundedness_of_deterministic_coefficients_white_noise_in_time}, Lemma \ref{prop_of_bessel_space} \eqref{bounded_operator}, and \eqref{pointwise_multiplier}, we have
\begin{equation*}
\begin{aligned}
&\bE \int_0^{\tau_{m,n}} \left\| \bar{b}^i(s,\cdot)  \left( (u(s,\cdot))^{1+\lambda} \right)_{x^i} \right\|_{H_q^{-1}}^q ds  \leq N \bE \int_0^{\tau_{m,n}} \int_{\bR^d}|u(s,x)|^{q(1+\lambda)} dxds < \infty.
\end{aligned}
\end{equation*}
additionally, according again to Lemma \ref{prop_of_bessel_space} \eqref{pointwise_multiplier}, we have
\begin{equation*}
\begin{aligned}
&a^{ij}u_{x^ix^j} \in \bH_{q}^{-2}(\tau_{m,n}), \quad b^{i} u_{x^i} \in \bH_{q}^{-1}(\tau_{m,n}), \quad\text{and}\quad c u \in \bL_{q}(\tau_{m,n}).
\end{aligned}
\end{equation*}
Therefore, since $\bL_{q}(\tau_{m,n})\subset \bH_{q}^{-1}(\tau_{m,n}) \subset \bH_{q}^{-2}(\tau_{m,n})$, Lemma \ref{prop_of_bessel_space} \eqref{norm_bounded} implies that
\begin{equation} \label{deterministic_part_of_u_infinite_noise_1}
a^{ij}u_{x^ix^j} + b^{i} u_{x^i} + c u + b^i\left(u^{1+\lambda} \right)_{x^i} \in \bH_{q}^{-2}(\tau_{m,n}).
\end{equation}
By \eqref{boundedness_of_stochastic_coefficients_white_noise_in_time_large_lambda}, we have
\begin{equation} \label{proof_of_claim_stochastic_part}
\|\sigma(u)\|^q_{\bL_q(\tau_{m,n},\ell_2)} = \bE\int_0^{\tau_{m,n}} \int_{\bR^d} \left(\sum_{k} \left| \sigma^k(s,x,u(s,x)) \right|^2\right)^{q/2} dxds \leq N\left\| u \right\|_{\bL_q(\tau_{m,n})}^q < \infty.
\end{equation}
Therefore, we have
\begin{equation} \label{stochastic_part_of_u_infinite_noise_1}
\sigma( u ) \in \bL_q(\tau_{m,n},\ell_2) \subset \bH_{q}^{-1}(\tau_{m,n},\ell_2).
\end{equation}
Due to \eqref{deterministic_part_of_u_infinite_noise_1} and \eqref{stochastic_part_of_u_infinite_noise_1}, the right-hand side of equation \eqref{burger's_eq_white_noise_in_time} is an $H_{p}^{-2}$-valued continuous function on $[0,\tau_{m,n}]$. Thus, $u$ is in $\cL_q(\tau_{m,n})$, and $u$ satisfies \eqref{burger's_eq_white_noise_in_time} for all $t\leq \tau_{m,n}$ almost surely with $u(0,\cdot) = u_0(\cdot)$. On the other hand, since $\bar{b}^i \left( u^{1+\lambda} \right)_{x^i} \in \bH_q^{-1}(\tau_{m,n})$ and $\sigma( u ) \in \bL_q(\tau_{m,n},\ell_2) $, Theorem \ref{theorem_nonlinear_case} implies that there exists $v\in \cH_q^{1}(\tau_{m,n})$ satisfying
\begin{equation}
\label{equation_in_proof_of_consistency_infinite_noise_1}
dv = \left(a^{ij}v_{x^ix^j} + b^i v_{x^i} + cv +  \bar{b}^i\left( u^{1+\lambda}  \right)_{x^i}  \right) dt + \sigma^k(u) dw_t^k, \quad 0<t\leq\tau_{m,n}\,; \quad v(0,\cdot)=u_0. 
\end{equation}
Note that in \eqref{equation_in_proof_of_consistency_infinite_noise_1}, $\bar{b}^i\left( u^{1+\lambda}  \right)_{x^i}$ and $\sigma^k(u)$ are used, instead of $\bar{b}^i\left( v^{1+\lambda}  \right)_{x^i}$ and $\sigma^k(v)$, respectively. Since $u\in\cL_q(\tau_{m,n})$ satisfies equation \eqref{equation_in_proof_of_consistency_infinite_noise_1}, $w := u-v\in \cL_q(\tau_{m,n})$ satisfies 
\begin{equation*}
dw = \left(a^{ij}w_{x^ix^j} + b^i w_{x^i} + cw \right) dt, \quad 0<t\leq\tau_{m,n}\,; \quad w(0,\cdot)=0. 
\end{equation*}
By Theorem \ref{theorem_nonlinear_case}, we have $u = v$ in $\cL_q(\tau_{m,n})$. Therefore, $u$ is in $\cH_q^{1}(\tau_{m,n})$. Note that $\bar{u}\in \cH_q^1(\tau_{m,n})$ satisfies equation \eqref{burger's_eq_white_noise_in_time}. By Theorem \ref{theorem_burgers_eq_white_noise_in_time_large_lambda}, we have $u = \bar{u}$ in $\cH_q^{1}(\tau_{m,n})$.  The theorem is proved.
\end{proof}

\vspace{2mm}

\section{Proof of the second case: \texorpdfstring{$\lambda\in(0,1/d)$}{Lg} and super-linear diffusion coefficient }
\label{proof_of_theorem_white_noise_in_time_small_lambda}

This section contains proofs of Theorems \ref{theorem_burgers_eq_white_noise_in_time_small_lambda} and \ref{uniqueness_in_p_2}. To provide motivation for the proof, we introduce a lemma. 
\begin{lemma}
\label{lemma: sol of eq with xi}
Let $r,r_0\in(d,\infty)$, $T\in(0,\infty)$, and $\kappa\in (d/r \vee d/r_0 , 1)$. Suppose Assumptions \ref{assumptions_on_coefficients_deterministic_and_linear_part}, \ref{assumptions_on_coefficients_deterministic_and_nonlinear_part} and \ref{stochastic_part_assumption_on_coeffi_white_noise_in_time_small_lambda} hold. Assume that there exists $K<\infty$ such that
\begin{equation}
\label{condition of xi,xi_0}
\sup_{s\leq T}\|\xi(s,\cdot)\|_{L_r} + \sup_{s\leq T}\|\xi_0(s,\cdot)\|_{L_{r_0}} < K  \quad \text{(a.s.)}.
\end{equation} 
Then, for a bounded stopping time $\tau\leq T$ and nonnegative initial data $u_0\in U_p^{1-\kappa}$, there exists $v\in\cH_p^{1-\kappa}(\tau)$ such that
\begin{equation}
\label{example equation}
dv = \left(a^{ij}v_{x^ix^j} + b^i v_{x^i} + cv + \bar{b}^i\left( \xi v \right)_{x^i}\right) dt + \mu^k\xi_0\,v\, dw_t^k,\quad t > 0\,; \quad v(0,\cdot) = u_0(\cdot).
\end{equation}
\end{lemma}
\begin{proof}
Observe that for $v_1, v_2\in \cH_p^{1-\kappa}(\tau)$, we have
\begin{equation*}
\begin{aligned}
& \left\| \bar b^i \left(\xi (v_1 - v_2) \right)_{x^i} \right\|_{\bH_{p}^{-1-\kappa}(\tau)}^p + N\left\| \mu \xi_0 (v_1 - v_2) \right\|_{\bH_{p}^{-\kappa}(\tau,\ell_2)}^p \\
&\leq N\left\|  \xi (v_1 - v_2) \right\|_{\bH_{p}^{-\kappa}(\tau)}^p + N\left\| \xi_0 (v_1 - v_2) \right\|_{\bH_{p}^{-\kappa}(\tau)}^p \\
& \leq  N\bE\int_0^{\tau}\int_{\bR^d} \left( \int_{\bR^d} R_{\kappa}(x-y)|\xi(s,y)||v_1(s,y) - v_2(s,y)| dy \right)^p dxds \\
&\quad +  N \bE\int_0^{\tau}\int_{\bR^d} \left( \int_{\bR^d} R_{\kappa}(x-y)|\xi_0(s,y)||v_1(s,y) - v_2(s,y)| dy \right)^p dxds\\
&\leq N\bE\int_0^{\tau}  \|\xi(s,\cdot)\|_{L_r}^p \int_{\bR^d} \left( \int_{\bR^d} |R_{\kappa}(x-y)|^{\frac{r}{r-1}}|v_1(s,y) - v_2(s,y)|^{\frac{r}{r-1}} dy \right)^{p(1-1/r)} dxds \\
& \quad + N\bE\int_0^{\tau} \|\xi(s,\cdot)\|_{L_{r_0}}^p \int_{\bR^d} \left( \int_{\bR^d} |R_{\kappa}(x-y)|^{\frac{r_0}{r_0-1}}|v_1(s,y) - v_2(s,y)|^{\frac{r_0}{r_0-1}} dy \right)^{p(1-1/r_0)} dxds \\
&\leq N\bE\int_0^{\tau}\left(  \| \xi(s,\cdot) \|_{L_r}^{p/r}\| R_{\kappa} \|_{L_{\frac{r}{r-1}}}^p + \| \xi_0(s,\cdot) \|_{L_{r_0}}^{p/r_0}\| R_{\kappa} \|_{L_{\frac{r_0}{r_0-1}}}^p \right)\int_{\bR^d}|v_1(s,x) - v_2(s,x)|^p dxds, \\
&\leq N\|v_1 - v_2 \|_{\bL_p(\tau)}^p, \\
&\leq \ep \| v_1 - v_2 \|_{\bH^{1-\kappa}_p(\tau)}^p +  N \| v_1 - v_2 \|_{\bH^{-\kappa}_p(\tau)}^p,
\end{aligned}
\end{equation*}
where $R_\kappa$ is the function introduced in Remark \ref{Kernel}. Note that Lemma \ref{prop_of_bessel_space} \eqref{multi_ineq} is used to obtain the last inequality. Then, by Theorem \ref{theorem_nonlinear_case}, the lemma is proved.

\end{proof}

To obtain the unique $L_p$-solvability of equation \eqref{burger's_eq_white_noise_in_time_small} using Lemma \ref{lemma: sol of eq with xi}, we consider equation \eqref{example equation} with $\xi = |u|^\lambda$ and $\xi_0 = |u|^{\lambda_0}$. As explained in Remark \ref{remark: discussion of regularity in super-linear case}, we estimate a uniform $L_1(\bR^d)$ bound for the solutions to obtain the unique $L_p$-solvability of equation \eqref{burger's_eq_white_noise_in_time_small}. Since the uniform $L_1(\bR^d)$ bound of $u$ is provided in Lemma \ref{L_1_bound}, if we set $s = 1/\lambda$ and $s_0 = 1/\lambda_0$ to control the coefficients $\xi = |u|^{\lambda}$ and $\xi_0 = |u|^{\lambda_0}$, then $\xi = |u|^\lambda$ and $\xi_0 = |u|^{\lambda_0}$ satisfy \eqref{condition of xi,xi_0}. Note that $\lambda$ and $\lambda_0$ should be smaller than $1/d$ since 
$1/\lambda = s > d$ and $1/\lambda_0 = s_0 > d$. In detail, for given $\kappa\in ((\lambda d) \vee (\lambda_0 d),1)$, \eqref{nonlinear_estimate} of Theorem \ref{theorem_nonlinear_case} and \eqref{main_computation_Non_explosion_small_lambda} imply
\begin{align*}
\|u\|_{\cH_p^{1-\kappa}(\tau)}
&\leq N\left(\|u_0\|_{U_p^{1-\kappa}} + \| \bar b^i \xi u \|_{\bH_p^{-3/2-\kappa}(\tau)} + \| \mu \xi_0 u \|_{H_p^{-1/2-\kappa}(\ell_2)}\right) \\
&\leq N\|u_0\|_{U_p^{1-\kappa}} + N\bE\int_0^\tau \left(\| \xi(s,\cdot) \|_{L_s}^{1/s} + \| \xi_0(s,\cdot) \|_{L_{s_0}}^{1/{s_0}}\right)\| u(s,\cdot) \|_{L_p}ds,
\end{align*}
where $s = 1/\lambda$, $s_0 = 1/\lambda_0$, and $N = N(d,p,\kappa,K,T)$. Since $\| \xi(s,\cdot) \|_{L_s} = \| \xi_0(s,\cdot) \|_{L_{s_0}} = \| u(s,\cdot) \|_{L_1}$ is uniformly bounded by the initial data $\|u_0\|_{L_1}$, we can extend $u$ to a global solution; see Lemma \ref{Non_explosion_small_lambda}. 
To demonstrate that $\| u(t,\cdot) \|_{L_1}$ is bounded for all $t\leq T$ almost surely, we apply the maximum principle and take the expectation to equation \eqref{burger's_eq_white_noise_in_time_small}. Then, we obtain that $\| u(t,\cdot) \|_{L_1}$ is a continuous local martingale, and thus, all the paths of $\| u(t,\cdot) \|_{L_1}$ are bounded.

This section is organized as follows. In Lemma \ref{cut_off_lemma_small_lambda}, we prove that there exist local solutions to equation \eqref{burger's_eq_white_noise_in_time_small}. To obtain a uniform $L_1(\bR^d)$ bound of the solution, an auxiliary function $\psi_{k}$ is introduced in Lemma \ref{test_function_for_small_lambda}. In Lemmas \ref{L_1_bound} and \ref{Non_explosion_small_lambda}, we show that $\bE\sup_{t\leq T}\|u(t,\cdot)\|^{1/2}_{L_1}$ can be controlled by the initial data $u_0$ and that the local existence time can be extended. At the end of this section, we provide the proofs of Theorems \ref{theorem_burgers_eq_white_noise_in_time_small_lambda} and \ref{uniqueness_in_p_2}.

\vspace{2mm}

Recall that $h(z)\in C^1(\bR)$ such that $h(z) = 1$ on $|z|\leq 1$ and $h(z) = 0$ on $|z|\geq2$. 

\begin{lemma} \label{cut_off_lemma_small_lambda}
Let $\lambda \in(0,\infty)$, $\lambda_0\in(0,\infty)$, $T\in(0,\infty)$, $p \geq 2$, and $h_m$ is the function introduced in \eqref{def of hm}. Suppose Assumptions \ref{assumptions_on_coefficients_deterministic_and_linear_part}, \ref{assumptions_on_coefficients_deterministic_and_nonlinear_part}, and \ref{stochastic_part_assumption_on_coeffi_white_noise_in_time_small_lambda} hold. Then, for a bounded stopping time $\tau\leq T$, $m\in\bN$, $\kappa\in [0,1)$, and nonnegative initial data $u_0\in U_{p}^{1-\kappa}$, there exists a unique $u_m \in \cH_p^{1-\kappa}(\tau)$ such that $u_m$ satisfies equation 
\begin{equation} 
\label{cut_off_equation_small_lambda}
du = \left(a^{ij}u_{x^ix^j} + b^i u_{x^i} + cu +  \bar{b}^i\left(u_+^{1+\lambda }h_m(u) \right)_{x^i}  \right) dt + \mu^k u_+^{1+\lambda_0}h_m(u) dw_t^k,\, (t,x)\in(0,\tau)\times\bR^d
\end{equation}
with $u(0,\cdot) = u_0(\cdot)$ and $u_m\geq0$. 
\end{lemma}
\begin{proof}
Note that \eqref{lipschitz_check_cut_off} and \eqref{boundedness_of_stochastic_coefficients_white_noise_in_time_small_lambda} imply that for $\omega\in\Omega, t>0$, and $u,v\in \bR$, we have
\begin{equation*}
\left| u_+^{1+\lambda }h_m(u) - v_+^{1+\lambda }h_m(v) \right| \leq N_m|u-v|
\end{equation*}
and
\begin{equation*}
\left| \mu(t,x) u_+^{1+\lambda_0}h_m(u) - \mu(t,x) v_+^{1+\lambda_0}h_m(v) \right|_{\ell_2} \leq N_m|u-v|,
\end{equation*}
where $\mu = (\mu^1,\mu^2,\dots)$. Then, by following the proof for Lemma \ref{cut_off_lemma_large_lambda}, there exists $u_m\in \cH_p^{1-\kappa}(\tau)$ satisfying \eqref{cut_off_equation_small_lambda} and $u_m\geq0$.

\end{proof}

Below, we introduce an auxiliary function that obtains a uniform $L_1$ bound of the local solutions. 

\begin{lemma} \label{test_function_for_small_lambda}
For $k = 1,2,\cdots$, set
\begin{equation*}
\psi_k(x):=\frac{1}{\cosh(|x|/k)}.
\end{equation*}
Suppose Assumptions \ref{assumptions_on_coefficients_deterministic_and_linear_part}, \ref{assumptions_on_coefficients_deterministic_and_nonlinear_part}, and \ref{stochastic_part_assumption_on_coeffi_white_noise_in_time_small_lambda} hold. Let $m\in\bN$ and $K$ be the constant described in the assumptions. Then, for all $\omega, x,t$, 
\begin{equation}
\label{negative}
\begin{aligned}
&\left(a^{ij}\psi_{k}\right)_{x^ix^j } - \left(b^i\psi_{k}\right)_{x^i} - \bar b^i \left(u_m\right)^\lambda h_m(u_m )\psi_{kx^i} + (c-4K)\psi_{k} \\
&\quad = a^{ij}\psi_{k x^ix^j} + \left(2a^{ij}_{x^j}-b^i-\bar b^i \left(u_m\right)^\lambda h_m(u_m) \right)\psi_{kx^i} + \left(a^{ij}_{x^ix^j}-b^i_{x_i}+c-4K\right)\psi_k \\
& \quad\leq
K\psi_k(x)\left[\frac{2}{k^2}+\frac{3+(2m)^\lambda}{k}-1 \right],
\end{aligned}
\end{equation}
where $N = N(d,K)$.
\end{lemma}
\begin{proof}
The proof follows from \cite[Lemma 5.5]{choi2021regularity}. The only difference is 
\begin{equation*}
\begin{aligned}
&\left(2a^{ij}_{x^j} - b^i - \bar b^i \left(u_m\right)^\lambda h_m(u_m)  \right)\psi_{kx^i}(x) \\
&\quad = -\left ( 2a^{ij}_{x^j} - b^i - \bar b^i \left(u_m\right)^\lambda h_m(u_m) \right)\psi_k(x)\frac{x^i\tanh(|x|/k)}{k|x|} \\
&\quad \leq k^{-1}\left(2|a^{ij}_{x^j}|+|b^i| + (2m)^\lambda |\bar b^i|\right)\psi_k(x).
\end{aligned}
\end{equation*}

\end{proof}

\begin{lemma}
\label{L_1_bound}
Let $u_m\in\cH_p^{1-\kappa}(T)$ be the solution introduced in Lemma \ref{cut_off_lemma_small_lambda}, and we further assume $u_0\in L_1(\Omega;L_1(\bR^d))$. Then, we have
\begin{equation} 
\label{L_1_bound_small_lambda}
\bE\sup_{t\leq T}\|u_m(t,\cdot)\|_{L_1(\bR^d)}^{1/2} \leq 3e^{2KT}\bE\|u_0\|_{L_1}^{1/2}.
\end{equation}
\end{lemma}
\begin{proof}
Let $\tau\leq T$ be a bounded stopping time and $\{\psi_k:k\in\bN\}$ be the sequence of functions introduced in Lemma \ref{test_function_for_small_lambda}. By multiplying $\psi_k$, employing It\^o's formula, using integration by parts, taking expectation, and applying inequality \eqref{negative}, 
we have
\begin{equation*}
\begin{aligned}
&\bE e^{-4K\tau}\int_{\bR^d}u_m(\tau,x)\psi_kdx - \bE\int_{\bR^d}u_0\psi_kdx \\
& =  \bE\int_0^\tau\int_{\bR^d}u_m\left[ a^{ij}\psi_{kx^ix^j} + \left(2a^{ij}_{x^j}-b^i\right)\psi_{kx^i} + \left(a^{ij}_{x^ix^j}-b^i_{x^i}+c-4K\right)\psi_{k} \right] dx e^{-4Ks} ds \\
&\quad - \bE\int_0^\tau\bar b^i(s)\int_{\bR^d}(u_{m+})^{1+\lambda}h_m(u_m)\psi_{kx^i} dx e^{-4Ks} ds \\
&\leq  N\bE\int_0^\tau \int_{\bR^d}u_m\psi_k\left[\frac{2}{k^2}+\frac{3+(2m)^\lambda}{k}-1 \right] dxds,
\end{aligned}
\end{equation*}
where $N = N(d,K)$. Note that there exists $M = M(m)$ such that if $k\geq M$, $\frac{2}{k^2}+\frac{3+(2m)^\lambda}{k}-1 \leq 0$. Thus, for $k\geq M$, we have
$$ \bE e^{-4K\tau}\int_{\bR^d}u_m(\tau,x)\psi_k(x)dx \leq \bE\| u_0 \|_{L_1}.
$$
Therefore, we have (e.g. \cite[Theorem III.6.8]{diffusion})
\begin{equation*}
\bE\sup_{t\leq T}\left(\int_{\bR^d}u_m(t,x) \psi_k(x)dx\right)^{1/2} \leq 3e^{2 KT}\bE\|u_0\|_{L_1}^{1/2}.
\end{equation*}
By letting $k\to\infty$, the monotone convergence theorem implies that 
\begin{equation} \label{L_1 est.}
\bE\sup_{t\leq T}\|u_m(t,\cdot)\|_{L_1}^{1/2} \leq 3e^{2KT}\bE\|u_0\|_{L_1}^{1/2}.
\end{equation}
The lemma is proved.
\end{proof}

\begin{lemma} \label{Non_explosion_small_lambda}
Let $\lambda ,\lambda_0\in(0,1/d)$, $T\in(0,\infty)$, $p > \frac{d+2}{1- (\lambda d) \vee (\lambda_0 d)}$, and $\kappa\in(\lambda d\vee\lambda_0 d,1)$. Let $u_m\in\cH_p^{1-\kappa}(T)$ be the solution to equation \eqref{cut_off_equation_small_lambda} introduced in Lemma \ref{cut_off_lemma_small_lambda}. Then, we have
\begin{equation} 
\label{stopping_time_blow_up}
\lim_{R\to\infty}\sup_mP\left( \left\{\omega\in\Omega : \sup_{t\leq T,x\in\bR^d}|u_m(t,x)|\geq R \right\} \right) = 0.
\end{equation}
\end{lemma}
\begin{proof}
For $m,S>0$, define
\begin{equation*}
\tau_m(S) :=\inf\{ t\geq 0:\| u_m(t,\cdot) \|_{L_1}\geq S \}\wedge T.
\end{equation*}
Due to \eqref{L_1_bound_small_lambda},  $\tau_m(S)$ is a well-defined stopping time. 
For $t\in(0,T)$, by Theorem \ref{theorem_nonlinear_case}, Lemma \ref{prop_of_bessel_space} \eqref{bounded_operator}, Remark \ref{Kernel}, H\"older's inequality, and Minkowski's inequality, we have
\begin{equation}
\label{main_computation_Non_explosion_small_lambda}
\begin{aligned}
&\| u_m \|^p_{\cH_{p}^{1-\kappa}(\tau_m(S)\wedge t)} - N\| u_0 \|^p_{U_p^{1-\kappa}}\\
&\leq N\left\| \bar b^i \left( u_m^{1+\lambda } \right)_{x^i} \right\|_{\bH_{p}^{-1-\kappa}(\tau_m(S)\wedge t)}^p + N\left\| \mu u_m^{1+\lambda_0} \right\|_{\bH_{p}^{-\kappa}(\tau_m(S)\wedge t,\ell_2)}^p \\
&\leq N\left\|  u_m^{1+\lambda } \right\|_{\bH_{p}^{-\kappa}(\tau_m(S)\wedge t)}^p + N\left\| u_m^{1+\lambda_0} \right\|_{\bH_{p}^{-\kappa}(\tau_m(S)\wedge t)}^p \\
&\leq  N\bE\int_0^{\tau_m(S)\wedge t}\int_{\bR^d} \left( \int_{\bR^d} R_{\kappa}(x-y)|u_m(s,y)|^{1+\lambda} dy \right)^p dxds\\
&\quad +  N \bE\int_0^{\tau_m(S)\wedge t}\int_{\bR^d} \left( \int_{\bR^d} R_{\kappa}(x-y)|u_m(s,y)|^{1+\lambda_0} dy \right)^p dxds\\
&\leq N\bE\int_0^{\tau_m(S)\wedge t} \|u_m(s,\cdot)\|^{p\lambda} \int_{\bR^d} \left( \int_{\bR^d} |R_{\kappa}(x-y)|^{\frac{1}{1-\lambda}}|u_m(s,y)|^{\frac{1}{1-\lambda}} dy \right)^{p(1-\lambda)} dxds \\
&\quad + N\bE\int_0^{\tau_m(S)\wedge t} \|u_m(s,\cdot)\|^{p\lambda_0} \int_{\bR^d} \left( \int_{\bR^d} |R_{\kappa}(x-y)|^{\frac{1}{1-\lambda_0}}|u_m(s,y)|^{\frac{1}{1-\lambda_0}} dy \right)^{p(1-\lambda_0)} dxds \\
&\leq N\bE\int_0^{\tau_m(S)\wedge t}\left(  \| u_m(s,\cdot) \|_{L_1}^{p\lambda}\| R_{\kappa} \|_{L_{\frac{1}{1-\lambda}}}^p + \| u_m(s,\cdot) \|_{L_1}^{p\lambda_0}\| R_{\kappa} \|_{L_{\frac{1}{1-\lambda_0}}}^p \right)\int_{\bR^d}|u_m(s,x)|^p dxds, \\
&\leq N\left(  S^{p\lambda}\| R_{\kappa} \|_{L_{\frac{1}{1-\lambda}}}^p + S^{p\lambda_0}\| R_{\kappa} \|_{L_{\frac{1}{1-\lambda_0}}}^p \right)\bE\int_0^{\tau_m(S)\wedge t}\int_{\bR^d}|u_m(s,x)|^p dxds,
\end{aligned}
\end{equation}
where $N = N(p,\kappa,K,T)$. Due to Remark \ref{Kernel} and $\kappa > \lambda d\vee \lambda_0 d$, we have $\| R_{\kappa} \|_{L_{\frac{1}{1-\lambda }}} + \| R_{\kappa} \|_{L_{\frac{1}{1-\lambda_0}}} < \infty$. Thus,
\begin{equation*}
\| u_m \|^p_{\cH_{p}^{1-\kappa}(\tau_m(S)\wedge t)} \leq N_1\| u_0 \|^p_{U_p^{1-\kappa}} + N_2\| u_m \|_{\bL_p(\tau_m(S)\wedge t)}^p,
\end{equation*}
where $N_1 = N_1(p,\kappa,K,T)$ and  $N_2 = N_2(S,\lambda ,\lambda_0,p,\kappa,K,T)$. By Corollary \ref{embedding_corollary} and Theorem \ref{embedding} \eqref{gronwall_type_ineq}, we have
\begin{equation*}
\bE\sup_{t\leq \tau_m(S)\wedge T,x\in\bR^d}|u_m(s,x)|^p \leq N\| u_m \|^p_{\cH_{p}^{1-\kappa}(\tau_m(S)\wedge T)} \leq N  \|u_0\|_{U_p^{1-\kappa}}^p,
\end{equation*}
where $N = N(S,\lambda ,\lambda_0,p,\kappa,K,T)$. 

By \eqref{L_1_bound_small_lambda} and Chebyshev's inequality, we have
\begin{equation*}
\begin{aligned}
P\left( \sup_{t\leq T}\| u_m(t,\cdot) \|_{L_1}\geq S \right) \leq \frac{1}{\sqrt S}\bE\sup_{t\leq T}\| u_m(t,\cdot) \|_{L_1}^{1/2}  \leq \frac{N}{\sqrt S},
\end{aligned}
\end{equation*}
where $N = N(K,T)$.
Therefore, Chebyshev inequality yields that
\begin{equation*}
\begin{aligned}
&P\left(\sup_{t\leq T,x\in\bR^d}|u_m(t,x)| > R\right) \\
&\quad\leq P\left(\sup_{t\leq \tau_m(S)\wedge T,x\in\bR^d}|u_m(t,x)| > R\right) + P(\tau_m(S)<T)  \\
&\quad\leq P\left(\sup_{t\leq \tau_m(S) \wedge T,x\in\bR^d}|u_m(t,x)| > R\right) + P\left( \sup_{t\leq T}\| u_m(t,\cdot) \|_{L_1}\geq S \right)\\
&\quad\leq \frac{N_1}{R^p} + \frac{N_2}{\sqrt S},
\end{aligned}
\end{equation*}
where $N_1 = N_1(S,\lambda ,\lambda_0,p,\kappa,K,T)$ and $N_2 = N_2(K,T)$. By taking the supremum with respect to $m$ and letting $R\to\infty$ and $S\to\infty$ in order, we obtain \eqref{stopping_time_blow_up}. The lemma is proved.
\end{proof}

\begin{proof}[\bf{Proof of Theorem \ref{theorem_burgers_eq_white_noise_in_time_small_lambda}}]

{\it Step 1. (Uniqueness).}  Follow  {\it Step 1} in the proof of Theorem \ref{theorem_burgers_eq_white_noise_in_time_large_lambda}. The only difference is the employment of $\cH_{p,loc}^{1-\kappa}$ as a solution space instead of $\cH_{p,loc}^{1}$.

{\it Step 2. (Existence).} The motivation for the proof follows from \cite[Section 8.4]{kry1999analytic} and \cite[Theorem 2.11]{han2019boundary}.
Let $\kappa\in((\lambda d)\vee(\lambda_0 d),1)$ and $T\in(0,\infty)$. According to Lemma \ref{cut_off_lemma_small_lambda}, there exists nonnegative $u_m\in\cH_{p}^{1-\kappa}(T)$ satisfying equation \eqref{cut_off_equation_small_lambda}. By Corollary \ref{embedding_corollary}, we have $u_m\in C([0,T];C(\bR^d))$ (a.s.) and
\begin{equation*}
\bE\sup_{t\leq T}\sup_{x\in\bR^d}|u_m(t,x)| < \infty
\end{equation*}
For $R\in\{ 1,2,\dots,m-1 \}$, define
\begin{equation} \label{stopping_time_taumm_1}
\tau_m^R 
:= \inf\left\{ t\geq0 : \sup_{x\in\bR^d} |u_m(t,x)|\geq R \right\}
\end{equation}
Then, $\tau_m^R$ is a well-defined stopping time. It should be noted that $\tau_R^R \leq \tau_m^m$. Indeed, since $\sup\limits_{x\in\bR^d}|u_m(t,x)|\leq R$ for $t\leq \tau_m^R$,  we have $u_m\wedge m=u_m\wedge m\wedge R  = u_m \wedge R$ for $t\leq\tau_m^R$.  Thus, both $u_m$ and $u_R$ satisfy
\begin{equation*}
du = \left(a^{ij}u_{x^ix^j} + b^i u_{x^i} + cu +  \bar{b}^i\left(u^{1+\lambda}h_R(u) \right)_{x^i}  \right) dt + \mu^k u^{1+\lambda_0}h_R(u)   dw_t^k, \quad0<t\leq\tau_m^R
\end{equation*}
with initial data $u(0,\cdot)=u_0$. On the other hand, since $R < m$, $u_R\wedge R = u_R\wedge R\wedge m = u_R\wedge m$ for $t\leq \tau_R^R$. Then $u_m$ and $u_R$ satisfy
\begin{equation*}
du = \left(a^{ij}u_{x^ix^j} + b^i u_{x^i} + cu +  \bar{b}^i\left(u^{1+\lambda}h_R(u) \right)_{x^i}  \right) dt + \mu^k u^{1+\lambda_0}h_R(u)   dw_t^k, \quad0<t\leq\tau_R^R
\end{equation*}
with initial data $u(0,\cdot)=u_0$. Thus, by the uniqueness result in Lemma \ref{cut_off_lemma_small_lambda}, $u_m=u_R$ in $\cH_p^{1-\kappa}\left((\tau_m^R\vee\tau_R^R)\wedge T\right)$. Note that $\tau_R^R=\tau_m^R\leq\tau_m^m$ (a.s.).  Indeed, for $t<\tau_m^R$, 
\begin{equation*}
\sup_{s\leq t}\sup_{x\in\bR^d}|u_R(s,x)|=\sup_{s\leq t}\sup_{x\in\bR^d}|u_m(s,x)|\leq R,
\end{equation*}
which implies  $\tau_m^R\leq\tau_R^R$. Similarly, we have $\tau_m^R\geq\tau_R^R$. In addition, since $m > R$, we have $\tau_m^m \geq \tau_m^R$. Therefore, we have $ \tau_R^R \leq \tau_m^m $. Now, by Lemma \ref{Non_explosion_small_lambda}, 
\begin{equation*}
\begin{aligned}
\limsup_{m\to\infty}  P\left( \tau_m < T \right) &= \limsup_{m\to\infty} P\left(\sup_{t\leq T,x\in\bR^d}|u_m(t,x)|\geq m\right) \\
&\leq \limsup_{m\to\infty}\sup_{n} P\left(\sup_{t\leq T,x\in\bR^d}|u_n(t,x)|\geq m\right)\to 0,
\end{aligned}
\end{equation*}
which implies $\tau_m\to \infty$ in probability. Since $\tau_m$ is increasing,  we conclude that $\tau_m\uparrow \infty$ (a.s.). Define $\tau_m := \tau_m^m\wedge m$ and
\begin{equation*}
u(t,x):=u_m(t,x)\quad \text{for}~ t\in[0,\tau_m].
\end{equation*}
Note that $u$ satisfies \eqref{burger's_eq_white_noise_in_time_small} for all $t \leq \tau_m$ because $|u(t)| = |u_m(t)|\leq m$ for $t\leq\tau_m$. Since $u = u_m \in \cH_p^{1-\kappa}(\tau_m)$ for any $m$, we have $u\in\cH_{p,loc}^{1-\kappa}$.

{\it Step 3. (H\"older regularity). } Let $T<\infty$. Then, since $u\in \cH_p^{1-\kappa}(T)$, Corollary \ref{embedding_corollary} implies \eqref{holder_regularity_main_theorem_small_lambda}. The theorem is proved.

\end{proof}

\begin{proof}[\bf{Proof of Theorem \ref{uniqueness_in_p_2}}]
The proof is similar to that for Theorem \ref{uniqueness_in_p_1}. The only difference is to employ $\cH_q^{1-\kappa}$ instead of $\cH_q^1$. Additionally, we can use Corollary \ref{embedding_corollary} since $q>p>\frac{d+2}{1-\kappa}$. The theorem is proved.
\end{proof}

\vspace{2mm}

\bibliographystyle{plain}

\end{document}